\documentclass[reqno,a4paper,11pt]{amsart}
\usepackage[all,poly]{xy}
\usepackage{amsfonts}
\usepackage{amssymb}
\usepackage{amsmath}
\usepackage{color}
\usepackage[pagebackref,colorlinks]{hyperref}
\usepackage{enumerate}
     
\theoremstyle{plain}
\newtheorem {lemma}{Lemma}
\newtheorem {proposition}[lemma]{Proposition}
\newtheorem {theorem}[lemma]{Theorem}
\newtheorem {corollary}[lemma]{Corollary}

\newtheorem {replacement lemma}[lemma]{Replacement Lemma}

\theoremstyle{definition}
\newtheorem{definition}[lemma]{Definition}

\newtheorem{remark}[lemma]{Remark}
\newtheorem {example}[lemma]{Example}
\newtheorem*{proofthmm1}{Proof of Theorem 26}

\newcommand{\N}{\mathbb{N}}
\newcommand{\X}{\langle X \rangle}

\newcommand{\GKdim}{\operatorname{GKdim}}
\newcommand{\nod}{\operatorname{nod}}
\newcommand{\st}{\operatorname{st}}
\newcommand{\ind}{\operatorname{ind}}

\title[Locally finite weighted Leavitt path algebras]{Locally finite weighted Leavitt path algebras}

\author{Raimund Preusser}
\address{Department of Mathematics,
University of Brasilia, Brazil}
\email{raimund.preusser@gmx.de}

\subjclass[2000]{16S10, 16W10, 16W50, 16D70} 
\keywords{Weighted Leavitt path algebra, locally finite, GK dimension, Noetherian}

\begin{document}

\begin{abstract} 
A group graded $K$-algebra $A=\bigoplus\limits_{g\in G} A_g$ is called {\it locally finite} if $\dim_K A_g < \infty$ for every $g\in G$. We characterise the weighted graphs $(E,w)$ for which the weighted Leavitt path algebra $L_K(E,w)$ is locally finite with respect to its standard grading. We also prove that the locally finite weighted Leavitt path algebras are precisely the Noetherian ones and that $L_K(E,w)$ is locally finite iff $(E,w)$ is finite and the Gelfand-Kirillov dimension of $L_K(E,w)$ equals $0$ or $1$. Further it is shown that a locally finite weighted Leavitt path algebra is isomorphic to a locally finite Leavitt path algebra and therefore is isomorphic to a finite direct sum of matrix algebras over $K$ and $K[X,X^{-1}]$.
\end{abstract}

\maketitle

\section{Introduction}
The weighted Leavitt path algebras (wLpas) were introduced by R. Hazrat in \cite{hazrat13}. They generalise the Leavitt path algebras (Lpas). While the Lpas only embrace Leavitt's algebras $L_K(1,1+k)$ of module type $(1,k)$ where $k\geq 0$, the wLpas embrace all of Leavitt's algebras $L_K(n,n+k)$ of module type $(n,k)$ where $n\geq 1$ and $k\geq 0$. In \cite{hazrat-preusser} linear bases for wLpas were obtained. They were used to classify the simple and graded simple wLpas and the wLpas which are domains. In \cite{preusser} the Gelfand-Kirillov dimension of a weighted Leavitt path algebra $L_K(E,w)$, where $K$ is a field and $(E,w)$ is a finite weighted graph, was determined. Further finite-dimensional wLpas were investigated. It was shown that a finite-dimensional wLpa over a field $K$ is isomorphic to a finite-dimensional Lpa and therefore is isomorphic to a finite direct sum of matrix algebras over $K$. In this article we investigate locally finite weighted Leavitt path algebras $L_K(E,w)$ where $(E,w)$ is a nonempty, row-finite weighted graph and $K$ is a field. The main result is Theorem \ref{thmm}.

The rest of this paper is organised as follows. In Section 2 we recall some standard notation which is used throughout the paper. In Section 3 we recall the definition of a wLpa and the basis result from \cite{hazrat-preusser}. In Section 4 we characterise the weighted graphs $(E,w)$ for which the weighted Leavitt path algebra $L_K(E,w)$ is locally finite with respect to its standard grading. Further we show that $L_K(E,w)$ is locally finite iff $(E,w)$ is finite and $\GKdim L_K(E,w)\leq 1$ (while $L_K(E,w)$ is finite-dimensional iff $(E,w)$ is finite and $\GKdim L_K(E,w)=0$). In Section 5 we show that a locally finite wLpa is isomorphic to a locally finite Lpa. It follows that a locally finite weighted Leavitt path algebra $L_K(E,w)$ is isomorphic to a finite direct sum of matrix algebras over $K$ and $K[X,X^{-1}]$ and therefore is Noetherian. In Section 6 we show that a Noetherian wLpa is locally finite. In the last section we summarise the main results of Sections 4-6. 
\section{Notation}
If $X$ is any set, then we denote by $\X$ the set of all nonempty words over $X$ and by $\overline{\X}$ the set of all words over $X$ including the empty word. Together with juxtaposition of words $\X$ is a semigroup and $\overline{\X}$ a monoid. Let $A,B\in \X$. Then $B$ is called a {\it subword of $A$} if there are $C,D \in \overline{\X}$ such that $A=CBD$, a {\it prefix of $A$} if there is a $C \in \overline{\X}$ such that $A=BC$ and a {\it suffix of $A$} if there is a $C \in \overline{\X}$ such that $A=CB$. 

Throughout the paper $K$ denotes a field. If $X$ is a set, then we denote by $K\X$ the free associative $K$-algebra generated by $X$ (i.e. the $K$-vector space with basis $\X$ which becomes a $K$-algebra by linearly extending the juxtaposition of words).
\section{Weighted Leavitt path algebras}

\begin{definition}[\sc Directed graph]\label{defdg}
A {\it directed graph} is a quadruple $E=(E^0,E^1,s,r)$ where $E^0$ and $E^1$ are sets and $s,r:E^1\rightarrow E^0$ maps. The elements of $E^0$ are called {\it vertices} and the elements of $E^1$ {\it edges}. If $e$ is an edge, then $s(e)$ is called its {\it source} and $r(e)$ its {\it range}. $E$ is called {\it row-finite} if $s^{-1}(v)$ is a finite set for any vertex $v$ and {\it finite} if $E^0$ and $E^1$ are finite sets.
\end{definition}

\begin{definition}[{\sc Double graph of a directed graph}]\label{defddg}
Let $E$ be a directed graph. The directed graph $E_d=(E_d^0, E_d^1, s_d, r_d)$ where $E_d^0=E^0$, $ E_d^1=E^1\cup (E^1)^*$ where $(E^1)^*=\{e^*\mid e\in E^1\}$,
\[s_d(e)=s(e),~r_d(e)=r(e),~s_d(e^*)=r(e) \text{ and }r_d(e^*)=s(e)\text{ for any }e\in E^1\]
is called the {\it double graph of $E$}. We sometimes refer to the edges in the graph $E$ as {\it real edges} and the additional edges in $E_d$ (i.e. the elements of $(E^1)^*$) as {\it ghost edges}. 
\end{definition}

\begin{definition}[{\sc Path}]
Let $E$ be a directed graph. A {\it path} is a word $p=x_1\dots x_n\in \langle E^0\cup E^1\rangle$ such that either $x_i\in E^1~(i=1,\dots,n)$ and $r(x_i)=s(x_{i+1})~(i=1,\dots,n-1)$ or $n=1$ and $x_1\in E^0$. By definition, the {\it length} $|p|$ of $p$ is $n$ in the first case and $0$ in the latter case. $p$ is called {\it nontrivial} if $|p|>0$ and {\it trivial} if $|p|=0$. We set $s(p):=s(x_1)$ and $r(p):=r(x_n)$ (here we use the convention $s(v)=v=r(v)$ for any $v\in E^0$).
\end{definition}

\begin{definition}[\sc Path algebra]\label{defpa}
Let $E$ be a directed graph. The quotient $K\langle E^0\cup E^1\rangle/I $ where $I$ is the ideal of $K\langle E^0\cup E^1\rangle$ generated by the relations
\begin{enumerate}[(i)]
\item $vw=\delta_{vw}v$ for any $v,w\in E^0$ and
\item $s(e)e=e=er(e)$ for any $e\in E^1$
\end{enumerate}
is called the {\it path algebra of $E$} and is denoted by $P_K(E)$.
\end{definition}

\begin{remark}
The paths in $E$ form a basis for the path algebra $P_K(E)$.
\end{remark}

\begin{definition}[{\sc Weighted graph}]\label{defsg}
A {\it weighted graph} is a pair $(E,w)$ where $E$ is a directed graph and $w:E^1\rightarrow \N=\{1,2,\dots\}$ is a map. If $e\in E^1$, then $w(e)$ is called the {\it weight} of $e$. A weighted graph $(E,w)$ is called {\it row-finite} (resp. {\it finite}) if $E$ is row-finite (resp. finite). In this article all weighted graphs are assumed to be row-finite and to have at least one vertex.
\end{definition}

\begin{remark}
Let $(E,w)$ be a weighted graph. In \cite{hazrat13} and \cite{hazrat-preusser}, $E^1$ was denoted by $E^{\st}$. What was denoted by $E^1$ in \cite{hazrat13} and \cite{hazrat-preusser} is denoted by $\hat E^1$ in this article (see the next definition).
\end{remark}

\begin{definition}[{\sc Directed graph associated to a weighted graph}]
Let $(E,w)$ be a weighted graph. The directed graph $\hat E=(\hat E^0, \hat E^1, \hat s, \hat r)$ where $\hat E^0=E^0$, $\hat E^1:=\{e_1,\dots,e_{w(e)}\mid e\in E^1\}$, $\hat s(e_i)=s(e)$ and $\hat r(e_i)=r(e)$ is called the {\it directed graph associated to $(E,w)$}. 
\end{definition}

In the following $(E,w)$ denotes a weighted graph. For a vertex $v\in E^0$ we set $w(v):=\max\{w(e)\mid e\in s^{-1}(v)\}$ with the convention $\max \emptyset=0$. $\hat E_d$ denotes the double graph of the directed graph $\hat E$ associated to $(E,w)$.

\begin{definition} [{\sc Weighted Leavitt path algebra}]\label{def3}
The quotient $P_K(\hat E_d)/I$ where $I$ is the ideal of $P_K(\hat E_d)$ generated by the relations
\begin{enumerate}[(i)]
\item $\sum\limits_{e\in s^{-1}(v)}e_ie_j^*= \delta_{ij}v$ for all $v\in E^0$ and $1\leq i, j\leq w(v)$ and
\medskip 

\item $\sum\limits_{1\leq i\leq w(v)}e_i^*f_i= \delta_{ef}r(e)$ for all $v\in E^0$ and $e,f\in s^{-1}(v)$
\end{enumerate}
is called {\it weighted Leavitt path algebra (wLpa) of $(E,w)$} and is denoted by $L_K(E,w)$. In relations (i) and (ii), we set $e_i$ and $e_i^*$ zero whenever $i > w(e)$. 
\end{definition}

\begin{remark}\label{grading}
Set $n:=\max\{w(e) \mid e \in E^{1}\}$. One can define a $\mathbb Z^n$-grading on $L_K(E,w)$ as follows (cf. \cite[Proposition 5.7]{hazrat13}). For $v\in E^0$ define $\deg(v)=0$ and for $e \in E^{1}$ and $1\leq i\leq w(e)$ define $\deg(e_i)=\epsilon_i$ and $\deg(e_i^*)=-\epsilon_i$. Here $\epsilon_i$ denotes the element of $\mathbb Z^n$ whose $i$-th component is $1$ and whose other components are $0$. This grading is called the {\it standard grading of $L_K(E,w)$}.
\end{remark}

\begin{example}\label{wlpapp}
It is easy to see that the wLpa of a weighted graph consisting of one vertex and $n+k$ loops of weight $n$ is isomorphic to the Leavitt algebra $L_K(n,n+k)$, for details see \cite[Example 4]{hazrat-preusser}. 
\end{example}

\begin{example}\label{exex1}
$(E,w)$ is called {\it unweighted} if $w(e)=1$ for all $e \in E^{1}$. It easy to see that if $(E,w)$ is unweighted, then $L_K(E,w)$ is graded isomorphic to the the usual Leavitt path algebra $L_K(E)$ (with respect to the standard gradings of $L_K(E,w)$ and $L_K(E)$) . 
\end{example}

We call a path in $\hat E_d$ a {\it d-path}. While the d-paths form a basis for the path algebra $P_K(\hat E_d)$, a basis for the weighted Leavitt path algebra $L_K(E,w)$ is formed by the nod-paths, which we will define in the next definition.

A vertex $v\in E^0$ is called a {\it sink} if $s^{-1}(v)=\emptyset$ and {\it regular} otherwise. For any regular vertex $v\in E^0$ fix an $e^{v}\in s^{-1}(v)$ such that $w(e^{v})=w(v)$. The $e^{v}$'s are called {\it special edges}. The words
\[e^v_i(e^v_j)^*~(v\in E^0,1\leq i,j\leq w(v))\text{ and }e^*_1f_1~(v\in E^0,e,f\in s^{-1}(v))\]
in $\langle \hat E_d^0\cup\hat E_d^1\rangle$ are called {\it forbidden}.   

\begin{definition}[{\sc Nod-path}]\label{deffnod}
A {\it normal d-path} or {\it nod-path} is a d-path such that none of its subwords is forbidden. 
\end{definition}

\begin{theorem}[Hazrat, Preusser, 2017] \label{thmhp}
Let $K\langle \hat E_d^0\cup\hat E_d^1\rangle_{\nod}$ be the linear span of all nod-paths in $K\langle \hat E_d^0\cup\hat E_d^1\rangle$. The canonical map $K\langle \hat E_d^0\cup\hat E_d^1\rangle_{\nod}\rightarrow L_K(E,w)$ is an isomorphism of $K$-vector spaces. In particular the images of the nod-paths in $L_K(E,w)$ form a basis for $L_K(E,w)$.
\end{theorem} 
\begin{proof}
See \cite[Theorem 16]{hazrat-preusser} and its proof.
\end{proof}

\section{Locally finite wLpas and their GK dimension}
A group graded $K$-algebra $A=\bigoplus\limits_{g\in G} A_g$ is called {\it locally finite} if $\dim_K A_g < \infty$ for every $g\in G$. In this section we characterise the weighted graphs whose wLpa is locally finite with respect to its standard grading. Further we show that $L_K(E,w)$ is locally finite iff $(E,w)$ is finite and $\GKdim L_K(E,w)\leq 1$. We start with some definitions. 

\begin{definition}[{\sc Closed path, cycle, entrance, exit}]
A {\it closed path (based at $v$)} is a nontrivial path $p$ such that $s(p)=r(p)=v$. A {\it cycle (based at $v$)} is a closed path $p=x_1\dots x_n$ based at $v$ such that $s(x_i)\neq s(x_j)$ for any $i\neq j$. An edge $e\in E^1$ is called {\it entrance} (resp. {\it exit}) of a cycle $x_1\dots x_n$ if there is an $i\in \{1,\dots,n\}$ such that $r(e)=r(x_i)$ (resp. $s(e)=s(x_i)$) and $e\neq x_i$.
\end{definition}

\begin{definition}[{\sc Tree, hereditary vertex set, edges in line}]\label{deftre}
If $u,v\in E^0$ and there is a path $p$ in $E$ such that $s(p)=u$ and $r(p)=v$, then we write $u\geq v$. Clearly $\geq$ is a preorder on $E^0$. If $u\in E^0$ then $T(u):=\{v\in E^0 \ |  \ u\geq v\}$ is called the {\it tree of $u$}. If $X\subseteq E^0$, we define $T(X):=\bigcup\limits_{v\in X}T(v)$. A subset $H\subseteq E^0$ is called {\it hereditary} if $T(H)\subseteq H$. Clearly any tree of a vertex is a hereditary subset of $E^0$. Two edges $e,f\in E^1$ are called {\it in line} if $e=f$ or $r(e)\geq s(f)$ or $r(e)\geq s(f)$
\end{definition}


\begin{definition}[{\sc Weighted subgraph defined by hereditary vertex set}]
Let $H\subseteq E^0$ be a hereditary subset. Set $E_H^0:=H$, $E_H^1:=\{e\in E^1\mid s(e)\in H\}$, $r_H:=r|_{E_H^1}$, $s_H=s|_{E_H^1}$ and $w_H:=w|_{E_H^1}$. Then $E_H:=(E_H^0,E_H^1,s_H,r_H)$ is a directed graph and $(E_H,w_H)$ a weighted graph. We call $(E_H,w_H)$ the {\it weighted subgraph of $(E,w)$ defined by $H$}.   
\end{definition}

\begin{definition}[{\sc Weighted and unweighted vertices and edges, weighted part of $(E,w)$}]
An element $v\in E^0$ (resp. $e\in E^{1}$) is called {\it weighted} if $w(v)>1$ (resp. $w(e)>1$) and {\it unweighted} otherwise. The subset of $E^0$ (resp. $E^1$) consisting of all weighted elements is denoted by $E_w^0$ (resp. $E^1_w$) and the subset of all unweighted elements by $E_{u}^0$ (resp. $E^1_{u}$). The weighted subgraph $(E_{T(E^{0}_w)},w_{T(E^{0}_w)})$ of $(E,w)$ defined by the hereditary subset $T(E^{0}_w)\subseteq E^0$ is called the {\it weighted part of $(E,w)$}.
\end{definition}

The next definition is based on an idea of P. N. Tan.

\begin{definition}[{\sc Well-behaved weighted part}]\label{defwb}
The weighted part of $(E,w)$ is called {\it weakly well-behaved} if the Conditions (i)-(iv) below hold.
\begin{enumerate}[(i)]
\item No vertex $v\in E_w^0$ emits two distinct weighted edges.
\item No vertex $v\in T(r(E^1_w))$ emits two distinct edges.
\item If $e,f\in E^1_w$ are not in line, then $T(r(e))\cap T(r(f))=\emptyset$.
\item No cycle in $(E,w)$ is based at a vertex $v\in T(r(E^1_w))$.
\end{enumerate}
The weighted part of $(E,w)$ is called {\it well-behaved} if it is weakly well-behaved and additionally Condition (v) below holds.
\begin{enumerate}[(v)]
\item There is no $n\geq 1$ and paths $p_1,\dots,p_n,q_1,\dots,q_n$ in $(E,w)$ such that $r(p_i)=r(q_i)~(1\leq i\leq n)$, $s(p_1)=s(q_n)$, $s(p_i)=s(q_{i-1})~(2\leq i \leq n)$ and for any $1\leq i \leq n$, the first letter of $p_i$ is a weighted edge, the first letter of $q_i$ is an unweighted edge and the last letters of $p_i$ and $q_i$ are distinct.
\end{enumerate}
\end{definition}

Each of the Conditions (i)-(v) in Definition \ref{defwb} above ``forbids" a certain constellation in the weighted graph $(E,w)$. The pictures below illustrate these forbidden constellations. A double headed arrow stands for a weighted edge (this notation is also an idea of P. N. Tan) and a dashed arrow stands for an unweighted edge. A dotted arrow stands for a path.
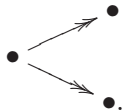
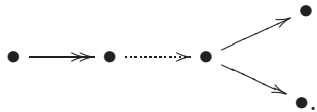
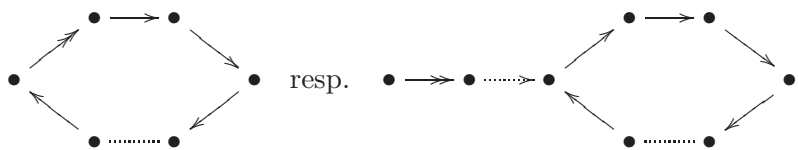
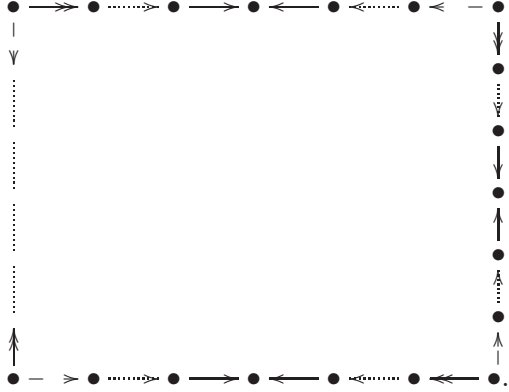
\begin{enumerate}[(i)]
\item \[\xymatrix@R-1.5pc{& \bullet\\\bullet \ar@{->>}[dr] \ar@{->>}[ur]& \\& \bullet.}\]
\item \[\xymatrix@R-1.5pc{& & & \bullet\\\bullet \ar@{->>}[r] & \bullet \ar@{.>}[r] & \bullet \ar[dr] \ar[ur]& \\& & & \bullet.}\]
\item \[\xymatrix{\bullet \ar@{->>}[r] & \bullet \ar@{.>}[r] & \bullet &\bullet \ar@{.>}[l] & \bullet. \ar@{->>}[l]}\]
\item \[\vcenter{\vbox{\xymatrix@R-1pc@C-0.5pc{
		&\bullet \ar[r]&\bullet \ar[rd]&\\
		\bullet \ar@{->>}[ru]&&&\bullet \ar[ld]\\
		&\bullet \ar[lu]&\bullet \ar@{.}[l]}}}
		~~~~\text{ resp. }~~~~
\vcenter{\vbox{\xymatrix@R-1pc@C-0.5pc{
		&&&\bullet \ar[r]&\bullet \ar[rd]&\\
		\bullet \ar@{->>}[r] & \bullet \ar@{.>}[r] &\bullet \ar[ru]&&&\bullet. \ar[ld]\\
		&&&\bullet \ar[lu]&\bullet \ar@{.}[l]}}}		
		\]
\item \[\xymatrix@R-1pc@C-0.5pc{
\bullet\ar@{-->}[d]\ar@{->>}[r]&\bullet\ar@{.>}[r]&\bullet\ar[r]&\bullet&\bullet\ar[l]&\bullet\ar@{.>}[l]&\bullet\ar@{-->}[l]\ar@{->>}[d]\\
\ar@{.}[d]&&&&&&\bullet\ar@{.>}[d]\\
\ar@{.}[d]&&&&&&\bullet\ar[d]\\
&&&&&&\bullet\\
\ar@{.}[u]&&&&&&\bullet\ar[u]\\
\ar@{.}[u]&&&&&&\bullet\ar@{.>}[u]\\
\bullet\ar@{->>}[u]\ar@{-->}[r]\ar@{.}[u]&\bullet\ar@{.>}[r]&\bullet\ar[r]&\bullet&\bullet\ar[l]&\bullet\ar@{.>}[l]&\bullet.\ar@{-->}[u]\ar@{->>}[l]
}\]
\end{enumerate}
$~$\\ 

We need two more definitions.
\begin{definition}[{\sc Index of a nontrivial nod-path}]
Let $p=x_1\dots x_n$ be a nontrivial nod-path. Then there are uniquely determined $e^{(i)}\in E^1,j_i\in \{1,\dots,w(e^{(i)})\}~(1\leq i\leq n)$ such that $x_i=e^{(i)}_{j_i}$ or $x_i=(e^{(i)}_{j_i})^*$ for any $1\leq i\leq n$. The positive integer $\ind(p):=\max\{j_i\mid 1\leq i\leq n\}$ is called the {\it index of $p$}. 
\end{definition}

\begin{definition}[{\sc Super-special and unweighted paths in $\hat E$}]
Let $p=e^{(1)}_{j_1}\dots e^{(n)}_{j_n}$, where $e^{(i)}\in E^1$ and $j_i\in \{1,\dots,w(e^{(i)})\}$ for any $1\leq i\leq n$, be a nontrivial path in $\hat E$. Then $p$ is called {\it super-special} if $j_i>\max \{w(f)\mid f\in s^{-1}(s(e^{(i)})), f\neq e^{(i)}\}$ (with the convention $\max \emptyset=0$) for any $1\leq i\leq n$. Note that if $p$ is super-special, then all $e^{(i)}$'s are special. $p$ is called {\it unweighted} if all $e^{(i)}$'s are unweighted.
\end{definition}

The next lemma describes a key property of weighted graphs whose weighted part is weakly well-behaved.
\begin{lemma}\label{lem3.-1}
Suppose the weighted part of $(E,w)$ is weakly well-behaved. Let $o=x_1\dots x_m$ be a nod-path such that $x_1=e_i$ for some $e\in E^1_w$ and $2\leq i\leq w(e)$. Then $o=p_1q_1^*\dots p_nq_n^*$ or $o=p_1q_1^*\dots p_{n-1}q_{n-1}^*p_n$ where $n\geq 1$, $p_1,\dots,p_n$ are super-special paths in $\hat E$ and $q_1,\dots,q_n$ are unweighted paths in $\hat E$.
\end{lemma}
\begin{proof}
Clearly there are paths $p_1,\dots,p_n$ and $q_1,\dots,q_n$ in $\hat E$ such that $o=p_1q_1^*\dots p_nq_n^*$ or $o=p_1q_1^*\dots p_{n-1}$ $q_{n-1}^*p_n$ (any nod-path starting with a real edge is of that form). We show by induction on $j$ that for any $1\leq j \leq n$, $p_j$ is super-special and $q_j$ is unweighted.\\
\\
\underline{$j=1$} Write $p_1=e^{(1)}_{i_1}\dots e^{(k)}_{i_k}$ (where $e^{(1)}=e$ and $i_1=i$) and $q_1=f^{(1)}_{j_1}\dots f^{(l)}_{j_l}$. It follows from Conditions (i) and (ii) in Definition \ref{defwb} that $p_1$ is super-special (since $e\in E^1_w$ and $2\leq i\leq w(e)$). Assume now that 
$q_1$ is not unweighted. Then there is an $1\leq t\leq l$ such that $f^{(t)}\in E^1_w$. It follows from Condition (iii) in Definition \ref{defwb} that $e^{(1)}=e$ and $f^{(t)}$ are in line. Now it is easy to deduce from Conditions (ii) and (iv) in Definition \ref{defwb} that $e^{(k)}=f^{(l)}$. But this yields a contradiction since $e^{(k)}$ is special and therefore $e^{(k)}_{i_k}(f^{(l)}_{j_l})^*$ is forbidden. Thus $q_1$ is unweighted.\\
\\
\underline{$j\rightarrow j+1$} Clearly the index of the first letter of $p_{j+1}$ is greater than $1$ since $q_j$ is unweighted. It follows as in the case $j=1$ that $p_j$ is super-special and $q_j$ is unweighted.
\end{proof}

Recall from \cite{preusser} that a {\it nod$^2$-path} is a nod-path $p$ such that $p^2$ is a nod-path. A {\it quasicycle} is a nod$^2$-path $p$ such that none of the subwords of $p^2$ of length $<|p|$ is a nod$^2$-path.
\begin{corollary}\label{cor3.0}
Suppose the weighted part of $(E,w)$ is well-behaved. Then $\{c,c^*\mid c \text{ is a cycle in }\hat E\}$ is the set of all quasicycles.  
\end{corollary}
\begin{proof}
Let $o=x_1\dots x_n$ be a quasicycle. First assume that $\ind(o)>1$. By \cite[Remark 16(b),(c)]{preusser} we may assume that $x_1=e_i$ for some $e\in E^1_w$ and $2\leq i\leq w(e)$. By Lemma \ref{lem3.-1}, $o=p_1q_1^*\dots p_nq_n^*$ or $o=p_1q_1^*\dots p_{n-1}q_{n-1}^*p_n$ where $n\geq 1$, $p_1,\dots,p_n$ are super-special paths in $\hat E$ and $q_1,\dots,q_n$ are unweighted paths in $\hat E$. But this contradicts Condition (v) in Definition \ref{defwb} (note that paths in $\hat E$ lift to paths in $(E,w)$). Hence $\ind(o)=1$. If $o$ contained both real and ghost edges, then $o^2$ would have a forbidden subword of the form $e_1^*f_1$ where $e,f\in E^1$ which is not possible since $o$ is a quasicycle. Thus $o=c$ or $o=c^*$ where $c$ is a cycle in $\hat E$. 
\end{proof}

$(E,w)$ is called {\it acyclic} if there is no cycle in $(E,w)$ and {\it aquasicyclic} if there is no quasicycle. In \cite{preusser} it was shown that $L_K(E,w)$ is finite-dimensional iff $(E,w)$ is finite and aquasicyclic. Below we give another characterisation of the finite-dimensional wLpas which is due to P. N. Tan. 
\begin{lemma}\label{lemfd}
$L_K(E,w)$ is finite-dimensional iff $(E,w)$ is finite, acyclic and its weighted part is well-behaved.
\end{lemma}
\begin{proof}
($\Rightarrow$) Suppose that $L_K(E,w)$ is finite-dimensional. It follows from Theorem \ref{thmhp} that $(E,w)$ is finite and acyclic. One checks easily that if one of the Conditions (i)-(v) in Definition \ref{defwb} were not satisfied, then there would be a nod$^2$-path (compare \cite[Proof of Lemma 33]{preusser}) and hence $L_K(E,w)$ would have infinite dimension, again by Theorem \ref{thmhp}. Hence the weighted part of $(E,w)$ is well-behaved.\\
($\Leftarrow$) Suppose that $(E,w)$ is finite and acyclic and its weighted part is well-behaved. Then $\hat E$ is acyclic since a cycle in $\hat E$ would lift to a cycle in $(E,w)$. It follows from Corollary \ref{cor3.0} that there is no quasicycle. Thus, by \cite[Theorem 46]{preusser}, $L_K(E,w)$ is finite-dimensional.
\end{proof}

Theorem \ref{thmfd} below follows from the lemma above and the fact that a finitely generated $K$-algebra $A$ is finite-dimensional as a $K$-vector space iff $\GKdim A=0$.
\begin{theorem}\label{thmfd}
The following are equivalent:
\begin{enumerate}[(i)]
\item $L_K(E,w)$ is finite-dimensional.
\item $(E,w)$ is finite, acyclic and the weighted part of $(E,w)$ is well-behaved.
\item $(E,w)$ is finite and $\GKdim L_K(E,w)=0$.
\end{enumerate}
\end{theorem}

The next goal is to prove the following theorem.
\begin{theorem}\label{thmm1}
The following are equivalent:
\begin{enumerate}[(i)]
\item $L_K(E,w)$ is locally finite.
\item $(E,w)$ is finite, no cycle has an exit and the weighted part of $(E,w)$ is well-behaved.
\item $(E,w)$ is finite and $\GKdim L_K(E,w)\leq 1$.
\end{enumerate}
\end{theorem}

We proceed by proving some results needed to prove Theorem \ref{thmm1}.
\begin{lemma}\label{lem3.0}
Suppose that $L_K(E,w)$ is locally finite. Then $(E,w)$ is finite. 
\end{lemma}
\begin{proof}
Assume that $(E,w)$ is not finite. Then $|E^0|=\infty$. But $E^0\subseteq L_K(E,w)_0$ and $E^0$ is a linearly independent set by Theorem \ref{thmhp} which contradicts the assumption that $(E,w)$ is locally finite.
\end{proof}
\begin{lemma}\label{lem3.1}
Suppose that $L_K(E,w)$ is locally finite. Then $\{p,p^*\mid p \text{ is a super-special, unweighted, closed}$ $\text{path in }\hat E\}$ is the set of all nod$^2$-paths.
\end{lemma}
\begin{proof}
Let $o=x_1\dots x_n$ be a nod$^2$-path. Then there are uniquely determined $e^{(i)}\in E^1,j_i\in \{1,\dots,w(e^{(i)})\}$ $(1\leq i\leq n)$ such that $x_i=e^{(i)}_{j_i}$ or $x_i=(e^{(i)}_{j_i})^*$ for any $1\leq i\leq n$. Assume that there is an $i$ such that $e^{(i)}\in E^1_w$. By \cite[Remark 16(b),(c)]{preusser} we may assume that $i=n$ and $x_n=(e^{(n)}_{j_n})^*$ (note that the statements of \cite[Remark 16(b),(c)]{preusser} still are true if one replaces any appearance of ``quasi-cycle" by ``nod$^2$-path"). Further we may assume that $j_n>1$ (if $j_n=1$, replace $o$ by the nod$^2$-path $o':=x_1\dots x_{n-1}(e^{(n)}_{2})^*$). Clearly for any $k\geq 1$, $o^k(o^*)^k$ is a nod-path in the homogeneous $0$-component $L_K(E,w)_0$. By Theorem \ref{thmhp} this contradicts the assumption that $L_K(E,w)$ is locally finite. Hence $e^{(1)},\dots,e^{(n)}\in E^1_u$. \\
Assume now that $|s^{-1}(s(e^{(i)}))|>1$ for some $1\leq i\leq n$. By \cite[Remark 16(b),(c)]{preusser} we may assume that $i=n$ and $x_n=e^{(n)}_{j_n}$. Choose $e^{s(e^{(n)})}\neq e^{(n)}$. Then for any $k\geq 1$, $o^k(o^*)^k$ is a nod-path in the homogeneous $0$-component $L_K(E,w)_0$. By Theorem \ref{thmhp} this contradicts the assumption that $L_K(E,w)$ is locally finite. Hence $s^{-1}(s(e^{(i)}))=\{e^{(i)}\}$ for any $1\leq i\leq n$.\\
Since $e^{(1)},\dots,e^{(n)}\in E^1_u$, $o$ cannot contain both real edges and ghost edges (otherwise $o^2$ would not be a nod-path). Hence $o=p$ or $o=p^*$ for some super-special, unweighted, closed path in $\hat E$. 
\end{proof}

\begin{corollary}\label{cor3.1}
Suppose that $L_K(E,w)$ is locally finite. Then no cycle in $(E,w)$ has an exit.
\end{corollary}
\begin{proof}
Let $c=e^{(1)}\dots e^{(n)}$ be a cycle in $(E,w)$. Then $\hat c:=e_1^{(1)}\dots e_1^{(n)}$ is a cycle in $\hat E$. It follows from Lemma \ref{lem3.1} that $s^{-1}(s(e^{(i)}))=\{e^{(i)}\}$ for any $1\leq i\leq n$. Hence $c$ has no exit.
\end{proof}

\begin{corollary}\label{cor3.2}
Suppose that $L_K(E,w)$ is locally finite. Then the weighted part of $(E,w)$ is well-behaved. 
\end{corollary}
\begin{proof}
One checks easily that if one of the Conditions (i)-(v) in Definition \ref{defwb} does not hold, then there is a nod$^2$-path starting with $e_j$ for some $e\in E_w^1$ and $2\leq j\leq w(e)$. But that contradicts Lemma \ref{lem3.1}.
\end{proof}

Let $E'$ denote the set of all real and ghost edges which do not appear in a quasicycle. Let $P'$ denote the set of all nod-paths which are composed from elements of $E'$.
\begin{lemma}\label{lem3.2}
Suppose that $(E,w)$ is finite, no cycle has an exit and the weighted part of $(E,w)$ is well-behaved. Let $c$ and $d$ be cycles in $\hat E$ based at vertices $u$ resp. $v$. Then the following hold.
\begin{enumerate}[(i)]
\item If $u=v$ than $c=d$.
\item There is no $p\in P'$ such that $\hat s_d(p)=u$ and $\hat r_d(p)=v$. 
\end{enumerate}
\end{lemma}
\begin{proof}
Clearly $c$ and $d$ have no exit in $\hat E$. Hence (i) holds.\\
Assume now that there is a $p=x_1\dots x_m\in P'$ such that $\hat s_d(p)=u$ and $\hat r_d(p)=v$. Since $c$ and $d$ have no exit, we have $x_1=e^*_1$ and $x_m=f_1$ where $e_1$ is an entrance for $c$ and $f_1$ is an entrance for $d$ (note that $w(e)=w(f)=1$ because of Condition (iv) in Definition \ref{defwb}). Let $i$ be minimal with the property that $x_{i}$ is a real edge. Then there are $g,h\in E^1$, $1\leq j \leq w(g)$ and $1\leq k\leq w(h)$ such that $x_{i-1}=g^*_j$ and $x_{i}=h_k$. Clearly $g\in E^1_u$ (and hence $j=1$) because of Condition (iv) in Definition \ref{defwb}. It follows that $h\in E^1_w$ and $k\geq 2$ since $g^*_jh_k$ is not forbidden. Set $o:=x_i\dots x_m$. By Lemma \ref{lem3.-1}, $o=p_1q_1^*\dots p_{n-1}q_{n-1}^*p_n$ where $n\geq 1$, $p_1,\dots,p_n$ are super-special paths in $\hat E$ and $q_1,\dots,q_{n-1}$ are unweighted paths in $\hat E$. Clearly the first letter of $p_n$ equals $h'_{k'}$ for some $h'\in E_w^1$ and $k'\geq 2$ (since $p_1$ has this property and $q_1,\dots,q_{n-1}$ are unweighted). But this contradicts Condition (iv) in Definition \ref{defwb} since $\hat r(p_n)= v$.
\end{proof}

In the following we will use the symbols $\overset{\nod}{\Longrightarrow}$, $\Longrightarrow$ and $\approx$. They are defined in \cite[Definition 14]{preusser} resp. \cite[Remark 16(b)]{preusser}. A quasicycle $q$ is called {\it selfconnected} if $q\overset{\nod}{\Longrightarrow} q$. A sequence $q_1,\dots,q_n$ of quasi-cycles is called a {\it chain of length $n$} if $q_i\not\approx q_j$ for any $i\neq j$ and $q_1 \Longrightarrow q_2\Longrightarrow \dots \Longrightarrow q_n$.
\begin{corollary}\label{cor3.3}
Suppose that $(E,w)$ is finite, no cycle has an exit and the weighted part of $(E,w)$ is well-behaved. Then $\GKdim L_K(E,w)\leq 1$.
\end{corollary}
\begin{proof}
By Corollary \ref{cor3.0} $\{c,c^*\mid c \text{ is a cycle in }\hat E\}$ is the set of all quasicycles. Since no cycle has an exit and the weighted part of $(E,w)$ is well-behaved, any cycle in $\hat E$ is unweighted and super-special. Suppose $qq'$ is a nod-path where $q$ and $q'$ are quasicycles. Then $q=q'$ by Lemma \ref{lem3.2}(i) (note that if $c$ is a cycle in $\hat E$, then $cc^*$ and $c^*c$ are not nod-paths since $c$ is unweighted and super-special). Next we show that if $q$ and $q'$ are quasicycles such that $q\not\approx q'$ or $q=q'$, then $q\overset{\nod}{\Longrightarrow} q'$ cannot hold (and hence no quasicycle is selfconnected and the maximal length of a chain of a quasicycles is $1$).\\
Let $q$ and $q'$ be quasicycles such that $q\not\approx q'$ or $q=q'$ and assume that $q\overset{\nod}{\Longrightarrow} q'$. We only consider the case that $q=c$ for some cycle $c$ in $\hat E$, the case that $q=c^*$ for some cycle $c$ in $\hat E$ can be treated similarly. Since $c\overset{\nod}{\Longrightarrow} q'$ there is a nod-path $o$ such that $c$ is not a prefix of $o$ and $coq'$ is a nod-path. Write $c=x_1\dots x_m$ and $o=y_1\dots y_n$.\\
\\
\underline{Case 1} Suppose that $y_1,\dots,y_n$ are real edges. Then $o=x_1\dots x_k$ for some $1\leq k <m$ since $c$ has no exit. Set $c':=x_{k+1}\dots x_mx_1\dots x_k$. Assume that $q'=d$ for some cycle $d$ in $\hat E$. Then it follows from Lemma \ref{lem3.2}(i) that $d=c'$. But that contradicts the assumption that $q\not\approx q'$ or $q=q'$. Assume now that $q'=d^*$ for some cycle $d$ in $\hat E$. It follows again from Lemma \ref{lem3.2}(i) that $d=c'$. But now we get the contradiction that $c'(c')^*$ is a nod-path.\\
\\
\underline{Case 2} Suppose that one of the letters $y_1,\dots,y_n$ is a ghost edge. Let $i$ be minimal such that $y_i$ is a ghost edge. Then $y_i=(f_j)^*$ for some $f\in E^1$ and $1\leq j\leq w(e)$. Clearly $y_1\dots y_{i-1}=x_1\dots x_{i-1}$. Clearly $x_{i-1}=e_1$ where $e\in E^1$ is special since $c$ is unweighted and super-special. Since $y_{i-1}y_i=x_{i-1}y_i=e_1(f_j)^*$ is not forbidden, we have $e\neq f$. Assume $f_j$ is a letter of a cycle. Then, by Lemma \ref{lem3.2}(i), $f_j$ is a letter of $c$. Since $r(e)=r(f)$ it follows that $f_j=e_1$ which cannot hold since $e\neq f$. Hence $f_j\in E'$. It follows from Lemma \ref{lem3.2}(ii), that $y_i\dots y_n\in P'$. But that contradicts Lemma \ref{lem3.2}(ii) since clearly there are cycles $c'$ and $d$ in $\hat E$ such that $c'$ is based at $\hat s_d(y_i\dots y_n)$ and $d$ is based at $\hat r_d(y_i\dots y_n)$.\\
\\
Hence no quasicycle is selfconnected and the maximal length of a chain of a quasicycles is $1$. Thus $\GKdim L_K(E,w)\leq 1$ by \cite[Theorem 22]{preusser}.
\end{proof}

\begin{lemma}\label{lem3.25}
Let $q$ be a quasicycle such that $qq^*$ is a nod-path and $q\approx q^*$. Then $q=q^*$.
\end{lemma}
\begin{proof}
Write $q=x_1\dots x_n$. Then $qq^*=x_1\dots x_nx_n^*\dots x_1^*$. Since $qq^*$ is a nod-path, $x_nx_n^*$ is not forbidden. Since $q\approx q^*$ there is a $k\in \{1,\dots,n\}$ such that $x_{k+1}\dots x_{n}x_{1}\dots x_{k}=x_n^*\dots x_1^*$. Assume that $k<n$. Then $x_l^*=x_n$ for some $l>1$. Since $x_nx_n^*$ is not forbidden, it follows that $x_n^*\dots x_l^*$ is a nod$^2$-path. But this cannot hold since $q^*$ is a quasicycle. Thus $k=n$ and therefore $q=q^*$.
\end{proof}

\begin{lemma}\label{lem3.3}
Suppose that $(E,w)$ is finite and $\GKdim L_K(E,w)\leq 1$. Then there is no quasicycle of homogeneous degree $0=(0,\dots,0)$.
\end{lemma}
\begin{proof}
Suppose that there is a quasicycle $q=x_1\dots x_n$ of homogeneous degree $0$. Assume that $\ind(q)=1$. Then $q$ consists only of real edges or only of ghost edges. But this contradicts the assumption that $q$ has homogeneous degree $0$. Therefore $\ind(q)>1$ and hence there is a $1\leq i \leq n$, an $e\in E_w^1$ and a $2\leq j \leq w(e)$ such that $x_i=e_j$ or $x_i=e_j^*$. By \cite[Remark 16(b),(c)]{preusser} we may assume that $i=n$ and $x_n=e_j^*$. It follows that $qq^*=x_1\dots x_{n-1}e_j^*e_j x_{n-1}^*\dots x_1^*$ is a nod-path and therefore $q\Longrightarrow q^*$. Hence, by \cite[Theorem 22]{preusser}, $q\approx q^*$. It follows from the previous lemma that $q=q^*$. Hence $x_1=e_j$. Set $q':=x_1\dots x_{n-1}e^*_1$. One checks easily that $q'$ is a quasicycle. Clearly $q(q')^*$ is a nod-path and therefore $q\Longrightarrow (q')^*$. But $q\not\approx (q')^*$ (by \cite[Remark 16(a)]{preusser}, $e^*_j$ appears precisely once in $q$; hence it does not appear in $q'$ and therefore $e_j$ does not appear in $(q')^*$).  This contradicts the assumption that $\GKdim L_K(E,w)\leq 1$.
\end{proof}

\begin{corollary}\label{cor3.4}
Suppose that $(E,w)$ is finite and $\GKdim L_K(E,w)\leq 1$. Then $L_K(E,w)$ is locally finite.
\end{corollary}
\begin{proof}
Since $(E,w)$ is finite and $\GKdim L_K(E,w)\leq 1$, any nontrivial nod-path is of the form $oq^kq'p$ where $q$ is a quasicycle, $k\geq 0$, $o$ and $p$ are either the empty word or nod-paths in $P'$ and $q'$ is either the empty word or a prefix of $q$ not equal to $q$. By \cite[Remark 16(a)]{preusser}, there are only finitely many quasicycles and, by Lemma \ref{lem3.3}, none of them is of homogeneous degree $0$. Further we have $|P'|<\infty$ by \cite[Lemma 21]{preusser}. Thus $L_K(E,w)$ is locally finite (for fixed $o$, $p$, $q$ and $q'$ as above there can at most be one nod-path of the form $oq^kq'p$ in each homogeneous component). 
\end{proof}

We are now ready to prove Theorem \ref{thmm1}.
\begin{proofthmm1}
(i)$\Rightarrow$(ii). Suppose $L_K(E,w)$ is locally finite. Then $(E,w)$ is finite, no cycle has an exit and the weighted part of $(E,w)$ is well-behaved by Lemma \ref{lem3.0}, Corollary \ref{cor3.1} and Corollary \ref{cor3.2}. \\
(ii)$\Rightarrow$(iii). Suppose that $(E,w)$ is finite, no cycle has an exit and the weighted part of $(E,w)$ is well-behaved. Then $\GKdim L_K(E,w)\leq 1$ by Corollary \ref{cor3.3}.\\
(iii)$\Rightarrow$(i). Suppose that $(E,w)$ is finite and $\GKdim L_K(E,w)\leq 1$. Then $L_K(E,w)$ is locally finite by Corollary \ref{cor3.4}.
\end{proofthmm1}

\begin{remark}
The proof of Corollary \ref{cor3.4} shows that if $L_K(E,w)$ is locally finite, then $|P'|^2kl$, where where $k$ is the number of quasicycles and $l$ is the maximal length of a quasicycle, is an upper bound for $\{\dim_K L_K(E,w)_g\mid g\in\mathbb{Z}^n\}$, where $n$ is the maximal weight.
\end{remark}

\section{Locally finite wLpas are isomorphic to locally finite Lpas}
In this section we show that if $L_K(E,w)$ is locally finite, then it is isomorphic to a locally finite Leavitt path algebra. It is well-known that a locally finite Leavitt path algebra over a field $K$ is isomorphic to a finite direct sum of matrix algebras over $K$ and $K[X,X^{-1}]$ and therefore is Noetherian, cf. \cite[Theorem 4.2.17]{abrams-ara-molina}.

We call an $e\in E^{1}_w$ {\it weighted edge of type A} if $s^{-1}(s(e))=\{e\}$ and {\it weighted edge of type B} otherwise. The subset of $E^{1}_w$ consisting of all weighted edges of type A (resp. B) is denoted by $E^1_{w,A}$ (resp. $E^1_{w,B}$).

\begin{proposition}\label{5.2}
Suppose $L_K(E,w)$ is locally finite. Then there is a weighted graph $(\tilde E,\tilde w)$ such that $\tilde E^{1}_w=\tilde E^{1}_{w,B}=E^{1}_{w,B}$, the elements of $\tilde r(\tilde E^{1}_w)$ are sinks in $(\tilde E, \tilde w)$, $L_K(\tilde E,\tilde w)\cong L_K(E,w)$ and $L_K(\tilde E,\tilde w)$ is locally finite.
\end{proposition}
\begin{proof}
Set $Z:=T(r(E^1_w))\cup s(E^1_{w,A})$. Note that $Z$ is a hereditary subset of $E^0$. Define a weighted graph $(\tilde E,\tilde w)$ by $\tilde E^0=E^0$, $\tilde E^1=\tilde E^1_Z\sqcup \tilde E^1_{Z^c}$ where
\[\tilde E^1_Z=\{e^{(1)},\dots,e^{(w(e))}\mid e\in E^1, s(e)\in Z\}\text{ and }\tilde E^1_{Z^c}=\{e\mid e\in E^1,s(e)\in E^0\setminus Z\},\]
$\tilde s(e^{(i)})=r(e)$, $\tilde r(e^{(i)})=s(e)$ and $\tilde w(e^{(i)})=1$ for any $e^{(i)}\in \tilde E^1_Z$ and $\tilde s(e)=s(e)$, $\tilde r(e)=r(e)$ and $\tilde w(e)=w(e)$ for any $e\in \tilde E^1_{Z^c}$. One checks easily that $\tilde E^{1}_w=\tilde E^{1}_{w,B}=E^{1}_{w,B}$ and the elements of $\tilde r(\tilde E^{1}_w)$ are sinks in $(\tilde E, \tilde w)$. The proof that $L_K(\tilde E,\tilde w)\cong L_K(E,w)$ is very similar to the proof of \cite[Proposition 28]{hazrat-preusser} and therefore is omitted.\\
It remains to show that $L_K(\tilde E,\tilde w)$ is locally finite. By Theorem \ref{thmm1} it suffices to show that $(\tilde E,\tilde w)$ is finite, no cycle in $(\tilde E,\tilde w)$ has an exit and the weighted part of $(\tilde E,\tilde w)$ is well-behaved. Clearly $(\tilde E,\tilde w)$ is finite. Next we show that no cycle in $(\tilde E,\tilde w)$ has an exit. Let $\tilde c=x_1\dots x_n$ be a cycle in $(\tilde E, \tilde w)$ and $y\in \tilde E^1$ an exit for $\tilde c$. First assume that $\tilde s(x_i)\not\in Z$ for any $1\leq i \leq n$. Then $x_1, \dots, x_n, y\in \tilde E^1_{Z^c}$ (note that $\tilde s(e^{(i)})=r(e)\in Z$ for any $e^{(i)}\in \tilde E^1_{Z}$ since $Z$ is hereditary). It follows that $\tilde c$ is a cycle in $(E,w)$ and $y$ is an exit for $\tilde c$ in $(E,w)$. But this yields a contradiction since no cycle in $(E,w)$ has an exit (by Theorem \ref{thmm1} since $L_K(E,w)$ is locally finite). Assume now that $\tilde s(x_i)\in Z$ for some $1\leq i \leq n$. Then $\tilde s(x_i)\in Z$ for any $1\leq i \leq n$ (note that $Z$ is also hereditary in $(\tilde E, \tilde w)$). Hence $x_1, \dots, x_n\in \tilde E^1_{Z}$. Define the map
\begin{align*}
\phi: \tilde E^1_Z&\rightarrow E^1\\
e^{(j)}&\mapsto e.
\end{align*}
Clearly $\phi(x_n)\dots \phi(x_1)$ is a cycle in $(E,w)$. But this yields a contradiction since $(E,w)$ satisfies Condition (iv) in Definition \ref{defwb}.\\
It remains to show that the weighted part of $(\tilde E,\tilde w)$ is well-behaved, i.e. that Conditions (i)-(v) in Definition \ref{defwb} are satisfied.
\begin{enumerate}[(i)]
\item If $v\in Z$, then clearly $\tilde s^{-1}(v)\subseteq \tilde E^1_Z$. But $\tilde w(e^{(i)})=1$ for any $e^{(i)}\in \tilde E^1_Z$. Hence $v$ does not emit any weighted edge in $(\tilde E, \tilde w)$. Suppose now that $v\in \tilde E^0\setminus Z$. One checks easily that $\tilde s^{-1}(v)=s^{-1}(v)$ and further $\tilde w(e)=w(e)$ for any $e\in \tilde s^{-1}(v)$. Since $v$ does not emit two distinct weighted edges in $(E,w)$, $v$ does not emit two distinct weighted edges in $(\tilde E, \tilde w)$.
\item Obvious since the ranges of the weighted edges in $(\tilde E, \tilde w)$ are sinks.
\item Assume that there are $e, f\in \tilde E^1_w$ such that $e$ and $f$ are not in line in $(\tilde E, \tilde w)$ (in particular $e\neq f$) and $\tilde T(\tilde r(e))\cap \tilde T(\tilde r(f))\neq \emptyset$ (if $z\in \tilde E^0$, then we denote by $\tilde T(z)$ the tree of $z$ in $(\tilde E, \tilde w)$). Since the ranges of $e$ and $f$ are sinks in $(\tilde E, \tilde w)$, it follows that $\tilde r(e)=\tilde r(f)=:v$. Clearly $e,f\in\tilde E^1_{Z^c}$. Hence $r(e)=\tilde r(e)=v=\tilde r(f)=r(f)$, $w(e)=\tilde w(e)>1$ and $w(f)=\tilde w(f)>1$. Since $(E,w)$ satisfies Condition (iii) in Definition \ref{defwb}, $e$ and $f$ are in line in $(E,w)$. Since $e\neq f$, we have $v=r(e)\geq s(f)$ or $v=r(f)\geq s(e)$. In either case the existence of a cycle based at $v$ follows. Since $(E,w)$ satisfies Condition (iv) in Definition \ref{defwb}, we arrived at a contradiction.
\item Obvious since the ranges of the weighted edges in $(\tilde E, \tilde w)$ are sinks.
\item Assume there is an $n\geq 1$ and paths $\tilde p_1,\dots,\tilde p_n,\tilde q_1,\dots,\tilde q_n$ in $(\tilde E,\tilde w)$ such that $\tilde r(\tilde p_i)=\tilde r(\tilde q_i)~(1\leq i\leq n)$, $\tilde s(\tilde p_1)=\tilde s(\tilde q_n)$, $\tilde s(\tilde p_i)=\tilde s(\tilde q_{i-1})~(2\leq i \leq n)$ and for any $1\leq i \leq n$, the first letter of $\tilde p_i$ is a weighted edge, the first letter of $\tilde q_i$ is an unweighted edge and the last letters of $\tilde p_i$ and $\tilde q_i$ are distinct. Clearly $\tilde s(\tilde p_i)\not\in Z$ and $\tilde r(\tilde p_i)\in Z$ for any $i$ since the first (and last since the ranges of weighted edges are sinks) letter of $\tilde p_i$ is weighted. Hence $\tilde s(\tilde q_i)\not\in Z$ and $\tilde r(\tilde q_i)\in Z$ for any $i$. For any $i$ write $\tilde q_i=\tilde f^{(i1)}\dots \tilde f^{(im_i)}$ where $\tilde f^{(i1)},\dots, \tilde f^{(im_i)}\in \tilde E^1$. Let $k_i$ be maximal with the property that $\tilde s(\tilde f^{(ik_i)})\not \in Z$. Then $\tilde s(\tilde f^{(i(k_i+1))}),\dots, \tilde s(\tilde f^{(im_i)})\in Z$ and hence $\tilde f^{(i(k_i+1))},\dots, \tilde f^{(im_i)}\in \tilde E^1_Z$. 
For any $1\leq i\leq n$ set $p_i:=\tilde p_i \phi(\tilde f^{(im_i)})\dots \phi(\tilde f^{(i(k_i+1))})$, where $\phi$ is the map defined on the previous page, and $q_i:=\tilde f^{(i1)}\dots \tilde f^{(ik_i)}$. One checks easily that $p_1,\dots,p_n,q_1,\dots,q_n$ are paths in $(E,w)$ such that $r(p_i)=r(q_i)~(1\leq i\leq n)$, $s(p_1)=s(q_n)$, $s(p_i)=s(q_{i-1})~(2\leq i \leq n)$ and for any $1\leq i \leq n$, the first letter of $p_i$ is a weighted edge, the first letter of $q_i$ is an unweighted edge and the last letters of $p_i$ and $q_i$ are distinct. But this yields a contradiction since $(E,w)$ satisfies Condition (v) in Definition \ref{defwb}.
\end{enumerate}
\end{proof}

\begin{example}\label{ex5.00}
Suppose $(E,w)$ is the weighted graph 
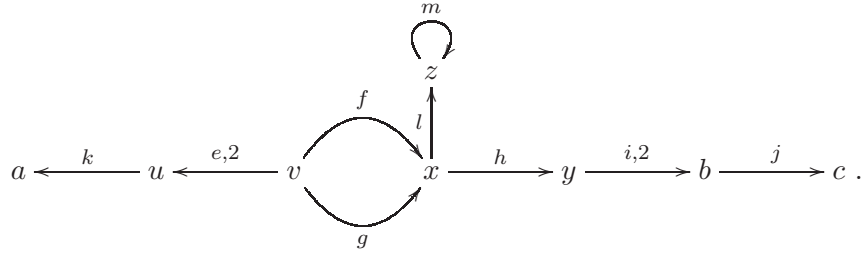
\[
\xymatrix@C+15pt{
&&&z\ar@(ul,ur)^{m}&&&\\
a&u\ar[l]_{k}&v\ar[l]_{e,2}\ar@/^1.7pc/[r]^{f}\ar@/_1.7pc/[r]_{g}&x\ar[u]^{l}\ar[r]^{h}&y\ar[r]^{i,2}&b\ar[r]^{j}&c~.}
\]
The weighted edge $i$ is of type A and the weighted edge $e$ is of type B. Let $Z$ be defined as in the proof of Proposition \ref{5.2}. Then $Z=\{a,u,y,b,c\}$. Let $(\tilde E, \tilde w)$ be the weighted graph
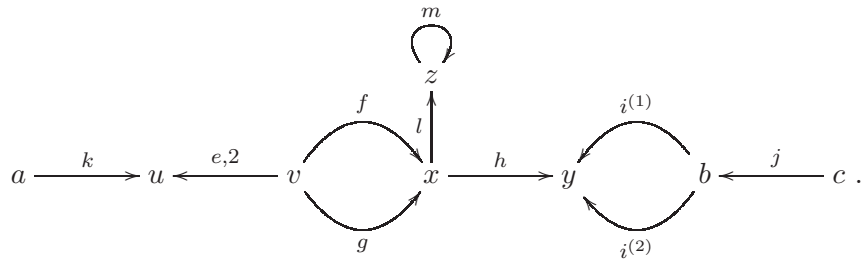
\[
\xymatrix@C+15pt{
&&&z\ar@(ul,ur)^{m}&&&\\
a \ar[r]^{k}& u& v\ar[l]_{e,2}\ar@/^1.7pc/[r]^{f}\ar@/_1.7pc/[r]_{g}& x\ar[u]^{l}\ar[r]^{h}& y& b\ar@/^1.7pc/[l]^{i^{(2)}}\ar@/_1.7pc/[l]_{i^{(1)}}& c\ar[l]_{j}~.}
\]
Then $L_K(E,w)\cong L_K(\tilde E, \tilde w)$. In $(\tilde E,\tilde w)$ there is only one weighted edge, namely $e$, and it is of type $B$. Further $u=\tilde r(e)$ is a sink in $(\tilde E, \tilde w)$.
\end{example}

The next goal is to remove also the weighted edges of type B without changing the wLpa, so that eventually one arrives at an unweighted weighted graph (see Example \ref{exex1}). 

\begin{lemma}\label{5.25}
Suppose $L_K(E,w)$ is locally finite and $(E,w)$ is not unweighted. Then there is a $v\in E^0_w$ such that $T(v)\cap E^0_w=\{v\}$ (i.e. $v$ is the only vertex in $T(v)$ which emits a weighted edge).
\end{lemma}
\begin{proof}
Let $\geq$ be the preorder on $E^0$ defined in Definition \ref{deftre}. Assume that there are distinct $u,v\in E^0_w$ such that $u\geq v$ and $v\geq u$. Then there is a closed path $p=e^{(1)}\dots e^{(n)}$ based at $u$. It follows from Lemma \ref{lem3.1} that $w(e^{(1)})=1$ and $s^{-1}(u)=\{e^{(1)}\}$. But this contradicts the assumption that $u\in E^0_w$. Hence the restriction of $\geq$ to $E^0_w$ is a partial order. Since $E^0_w$ is nonempty and finite, it contains a minimal element $v$. Clearly $T(v)\cap E^0_w=\{v\}$. 
\end{proof}

\begin{lemma}\label{5.3}
Suppose $L_K(E,w)$ is locally finite, $E^{1}_w=E^{1}_{w,B}$ and the elements of $r(E^{1}_w)$ are sinks. Let $v\in E^0_w$ such that $T(v)\cap E^0_w=\{v\}$. Then there is a directed graph $E'$ such that $L_K(E_{T(v)}, w_{T(v)})\cong L_K(E')$ via an isomorphism which maps vertices to sums of distinct vertices.
\end{lemma}
\begin{proof}
Clearly there is an integer $k\geq 2$, integers $m, n_1,\dots,n_m\geq 1$, a vertex $u\in E^0\setminus\{v\}$, pairwise distinct vertices $x_1,\dots,x_m\in E^0\setminus\{v\}$ and pairwise distinct edges $e, f^{(ij)}\in E^1~(1\leq i \leq m, 1\leq j\leq n_i)$ such that $s^{-1}(v)=\{e, f^{(ij)}\mid 1\leq i \leq m, 1\leq j\leq n_i\}$, $r(e)=u$, $r(f^{(ij)})=x_i$, $w(e)=k$ and $w(f^{(ij)})=1$. Set $X:=\{x_1,\dots,x_m\}$. Since no cycle in $(E,w)$ has an exit, $v\not\in T(X)$. Because of Condition (v) in Definition \ref{defwb} (applied with $n=1$), $u\not \in T(X)$. The picture below illustrates the weighted graph $(E_{T(v)}, w_{T(v)})$.
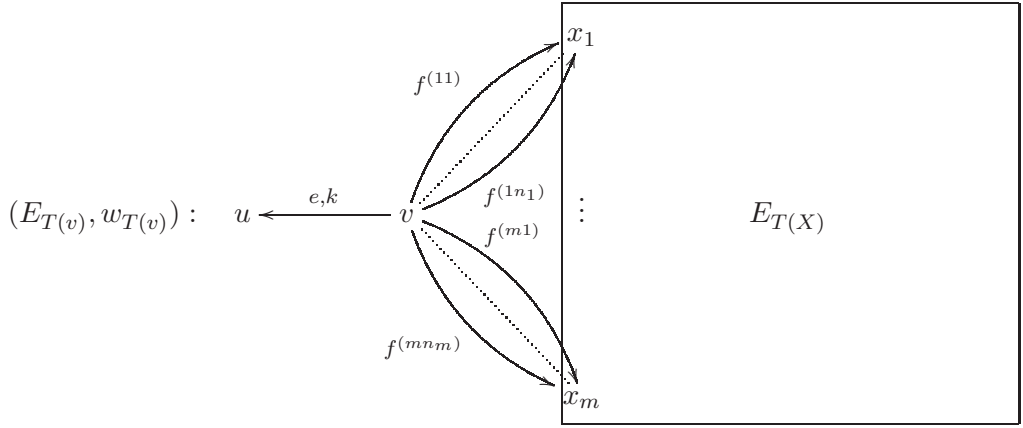
\[\xymatrix@C+25pt@R+25pt{
&&x_1&&
\save "1,3"."3,5"*+<1pc>!<-0.15pc,-0.5pc>\frm{-}\\
(E_{T(v)}, w_{T(v)}):\quad u&v\ar[l]_(0.30){e,k}\ar@/^1pc/[ur]^{f^{(11)}}\ar@{..}[ur]\ar@/_1pc/[ur]_(.33){f^{(1n_1)}}\ar@/^1pc/[dr]^(.33){f^{(m1)}}\ar@{..}[dr]\ar@/_1pc/[dr]_{f^{(mn_m)}}&\vdots&E_{T(X)}&\\
&&x_m&&
\restore}\]
Define a directed graph $E'$ by 
\begin{align*}
(E')^0:=&\{u_i~(1\leq i\leq k),u_{ij}~(1\leq i\leq m, 1\leq j \leq (k-1)n_i),v, v_{ij}~(1\leq i\leq m, 2\leq j \leq n_i)\}\sqcup T(X),\\
(E')^1:=&\{\alpha^{(i)}~(1\leq i\leq k),\beta^{(ij)}~(1\leq i\leq m, 1\leq j \leq n_i),\gamma^{(ij)}~(1\leq i\leq m, 1\leq j \leq (k-1)n_i)\}\sqcup s^{-1}(T(X)),\\
s'(\alpha^{(1)})=&v,~r'(\alpha^{(1)})=u_1, ~s'(\alpha^{(i)})=u_{i-1},~r'(\alpha^{(i)})=u_i~(i\geq 2),\\
s'(\beta^{(i1)})=&v,~r'(\beta^{(i1)})=v_{i2},~ s'(\beta^{(ij)})=v_{ij},~r'(\beta^{(ij)})=v_{i,j+1}~(1<j<n_i),~s'(\beta^{(in_i)})=v_{in_i},~r'(\beta^{(in_i)})=x_{i},\\
s'(\gamma^{(i1)})=&x_i,~r'(\gamma^{(i1)})=u_{i1}, ~s'(\gamma^{(ij)})=u_{i,j-1},~r'(\gamma^{(ij)})=u_{ij}~(j\geq 2),\\
s'(g)=&\begin{cases}s(g), \text{ if } s(g)\in T(X)\setminus X, \\ u_{i,(k-1)n_i}, \text{ if } s(g)=x_i,\end{cases},~r'(g)=\begin{cases}r(g), \text{ if } r(g)\in T(X)\setminus X, \\ u_{i,(k-1)n_i}, \text{ if } r(g)=x_i,\end{cases}~(g\in E^1, s(g)\in T(X)).
\end{align*}
The next picture illustrates $E'$. We set $p_{\alpha}:=\alpha^{(1)}\dots\alpha^{(k)}$, $p^{(i)}_{\beta}:=\beta^{(i1)}\dots\beta^{(in_i)}~(1\leq i\leq m)$ and $p^{(i)}_{\gamma}:=\gamma^{(i1)}\dots\gamma^{(i,(k-1)n_i)}~(1\leq i\leq m)$. By $E_{T(X)}^{x_i\leftrightarrow u_{(k-1)n_i}}$ we denote the directed graph one gets if one replaces $x_i$ by $u_{(k-1)n_i}$ in $E_{T(X)}$ for any $1\leq i \leq m$; in $E_{T(X)}^{x_i\leftrightarrow u_{(k-1)n_i}}$ the vertex $u_{(k-1)n_i}$ emits (resp. receives) the same edges that $x_i$ emits (resp. receives) in $E_{T(X)}$.
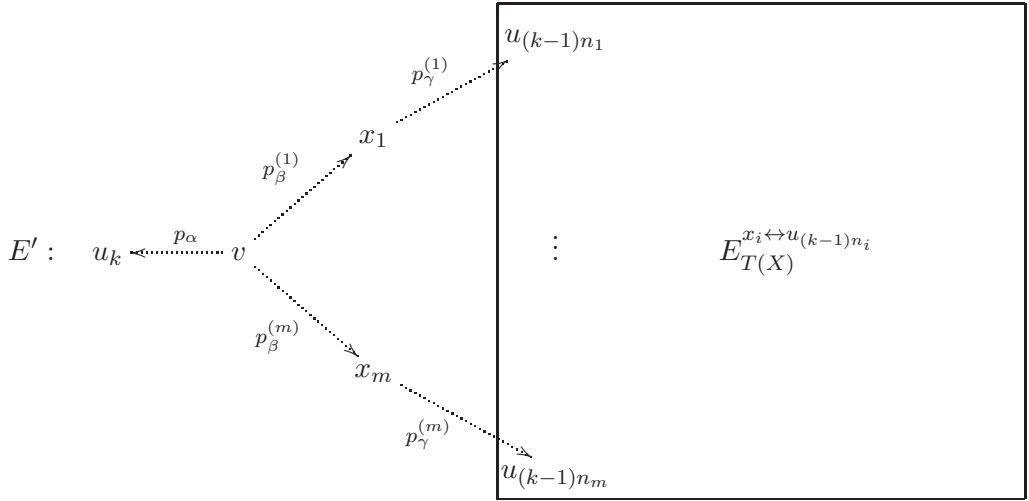
\[\xymatrix@C+10pt@R-0pt{
&&&u_{(k-1)n_1}&&
\save "1,4"."5,6"*+<1pc>!<-0.15pc,-0.5pc>\frm{-}\\
&&x_1\ar@{..>}[ur]^(.40){p^{(1)}_{\gamma}}&&&\\
E':\quad u_k&v\ar@{..>}[l]_(.30){p_{\alpha}}\ar@{..>}[ur]^{p^{(1)}_{\beta}}\ar@{..>}[dr]_{p^{(m)}_{\beta}}&&\vdots&E_{T(X)}^{x_i\leftrightarrow u_{(k-1)n_i}}&\\
&&x_m\ar@{..>}[dr]_(.40){p^{(m)}_{\gamma}}&&&\\
&&&u_{(k-1)n_m}&&
\restore
}
\]
There is an algebra isomorphism $\phi:L_K(E_{T(v)}, w_{T(v)})\rightarrow L_K(E')$ such that 
\begin{align*}
\phi(u)=\sum u_i+\sum u_{ij},~\phi(v)&=v+\sum v_{ij}\text{ and }\phi(y)=y~(y\in T(X)),
\end{align*}
cf. \cite[Proof of Lemma 37]{preusser}.
\end{proof}

Next we want to show that we can ``replace" the subgraph $(E_{T(v)}, w_{T(v)})$ in the previous lemma by the unweighted graph $(E', w')$ within $(E,w)$ without changing the wLpa. 
\begin{definition}[{\sc Replacement graph}]
Let $H\subseteq E^0$ a hereditary subset, $E'$ be a directed graph and $\phi:L_K(E_H,w_H)\rightarrow L_K(E')$ an isomorphism which maps vertices to sums of distinct vertices, i.e. for any $v\in H$ there are distinct $u'_{v,1},\dots,u'_{v,n_v}\in (E')^0$ such that $\phi(v)=u'_{v,1}+\dots+u'_{v,n_v}$. The weighted graph $(\tilde E, \tilde w)$ defined by   
\begin{align*}
&\tilde E^0=E^0\setminus H\sqcup (E')^0,\\
&\tilde E^{1}=\{e\mid e\in E^{1}, s(e),r(e)\in E^0\setminus H\}\\
&\quad\quad\sqcup\{e^{(1)},\dots,e^{(n_{r(e)})}\mid e\in E^{1}, s(e)\in E^0\setminus H,r(e)\in H\}\\
&\quad\quad\sqcup (E')^{1},\\
&\tilde s(e)=s(e),\tilde r(e)=r(e),\tilde w(e)= w(e)\quad(e\in E^{1}, s(e),r(e)\in E^0\setminus H),\\
&\tilde s(e^{(j)})=s(e),\tilde r(e^{(j)})=u'_{r(e),j},\tilde w(e^{(j)})= w(e)\quad(e\in E^{1}, s(e)\in E^0\setminus H,r(e)\in H ,1\leq j\leq n_{r(e)}),\\
&\tilde s(e')=s'(e'),\tilde r(e')=r'(e'),\tilde w(e')=1\quad(e'\in (E')^{1})
\end{align*}
is called the {\it replacement graph defined by $\phi$}.
\end{definition}
\begin{replacement lemma}\label{5.4}
Let $H\subseteq E^0$ be a hereditary subset, $E'$ a directed graph and $\phi:L_K(E_H,w_H)\rightarrow L_K(E')$ an isomorphism which maps vertices to sums of distinct vertices. Then $L_K(E,w)\cong L_K(\tilde E, \tilde w)$ where $(\tilde E, \tilde w)$ is the replacement graph defined by $\phi$.
\end{replacement lemma}
\begin{proof}
See \cite[Proof of Lemma 44]{preusser}.
\end{proof}
\begin{example}\label{ex5.2}
Suppose $(E,w)$ is the weighted graph 
\[
\xymatrix@C+15pt{
&&&z\ar@(ul,ur)^{m}&&&\\
a \ar[r]^{k}& u& v\ar[l]_{e,2}\ar@/^1.7pc/[r]^{f}\ar@/_1.7pc/[r]_{g}& x\ar[u]^{l}\ar[r]^{h}& y& b\ar@/^1.7pc/[l]^{i^{(2)}}\ar@/_1.7pc/[l]_{i^{(1)}}& c\ar[l]_{j}~.}
\]
Then $(E_{T(v)}, w_{T(v)})$ is the weighted graph
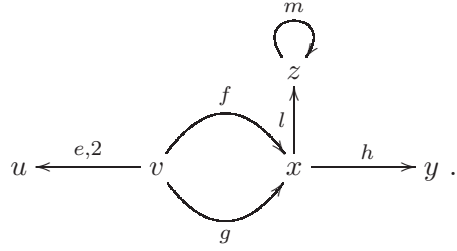
\[
\xymatrix@C+15pt{
&&z\ar@(ul,ur)^{m}&\\
u& v\ar[l]_{e,2}\ar@/^1.7pc/[r]^{f}\ar@/_1.7pc/[r]_{g}& x\ar[r]^{h}\ar[u]^{l}& y~.}
\]
Let $E'$ be the directed graph
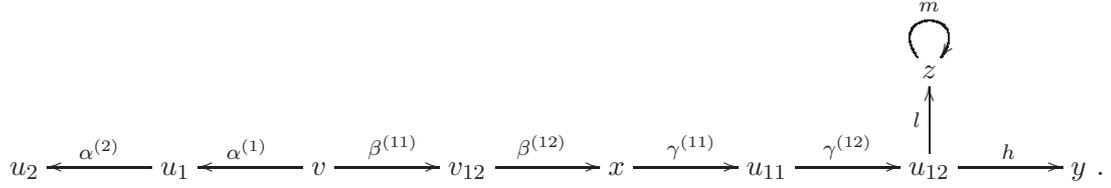
\[
\xymatrix@C+15pt{
&&&&&&z\ar@(ul,ur)^{m}&\\
u_2&u_1\ar[l]_{\alpha^{(2)}}&v\ar[l]_{\alpha^{(1)}}\ar[r]^{\beta^{(11)}}&v_{12}\ar[r]^{\beta^{(12)}}&x\ar[r]^{\gamma^{(11)}}&u_{11}\ar[r]^{\gamma^{(12)}}&u_{12}\ar[u]^{l}\ar[r]^{h}&y~.}
\]
Then $L_K(E_{T(v)}, w_{T(v)})\cong L_K(E')$ by Lemma \ref{5.3}. Let $\phi$ be the isomorphism mentioned in the proof of Lemma \ref{5.3}. Then the replacement graph defined by $\phi$ is the unweighted weighted graph 
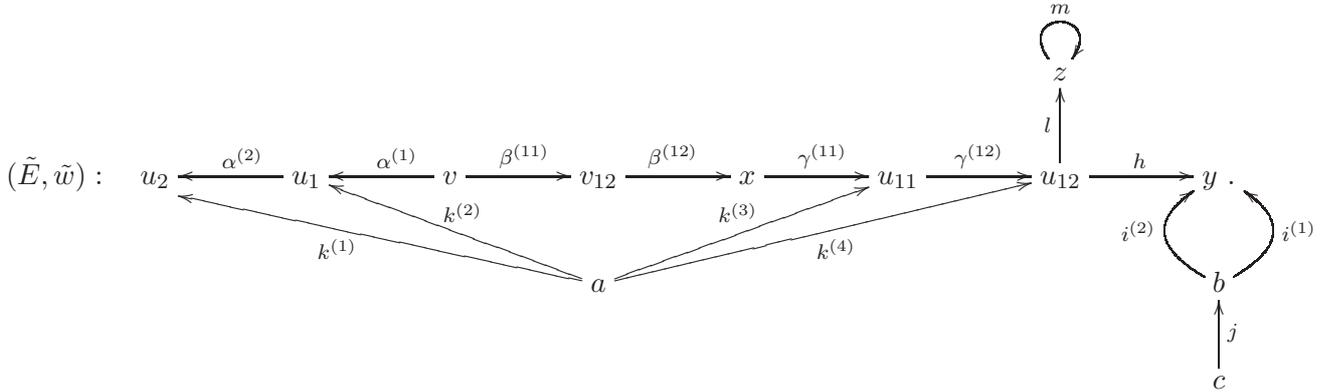
\[
\quad
\xymatrix@C+15pt{&&&&&&z\ar@(ul,ur)^{m}&\\
(\tilde E, \tilde w):\quad u_2&u_1\ar[l]_{\hspace{1.2cm}\alpha^{(2)}}&v\ar[l]_{\hspace{0.5cm}\alpha^{(1)}}\ar[r]^{\beta^{(11)}}&v_{12}\ar[r]^{\beta^{(12)}}&x\ar[r]^{\gamma^{(11)}}&u_{11}\ar[r]^{\gamma^{(12)}}&u_{12}\ar[u]^{l}\ar[r]^{h}&y~.\\
&&&a\ar[lllu]^{k^{(1)}}\ar[llu]_{k^{(2)}}\ar[rru]^{k^{(3)}}\ar[rrru]_{k^{(4)}}&&&&b\ar@/^1.7pc/[u]^{i^{(2)}}\ar@/_1.7pc/[u]_{i^{(1)}}\\
&&&&&&&c\ar[u]_{j}}
\]
By the Replacement Lemma we have $L_K(E,w)\cong L_K(\tilde E, \tilde w)$.
\end{example}

\begin{proposition}\label{5.5}
Suppose $L_K(E,w)$ is locally finite, $E^{1}_w=E^{1}_{w,B}$ and the elements of $r(E^{1}_w)$ are sinks. If $(E,w)$ is not unweighted, then there is a weighted graph $(\tilde E, \tilde w)$ such that $|\tilde E^1_w|=|E^1_w|-1$, $L_K(\tilde E,\tilde w)\cong L_K(E,w)$ and $L_K(\tilde E,\tilde w)$ is locally finite.
\end{proposition}
\begin{proof}
By Lemma \ref{5.25} there is a $v\in E^0_w$ such that $T(v)\cap E^0_w=\{v\}$. Let $E'$ be the directed graph defined in the proof of Lemma \ref{5.3}. Let $\phi:L_K(E_{T(v)}, w_{T(v)})\rightarrow L_K(E')$ be an isomorphism such that 
\begin{align*}
\phi(u)=\sum u_i+\sum u_{ij},~\phi(v)&=v+\sum v_{ij}\text{ and }\phi(y)=y~(y\in T(X))
\end{align*}
(such a $\phi$ exists, see the proof of Lemma \ref{5.3}). Let $(\tilde E, \tilde w)$ be the replacement graph defined by $\phi$. Then clearly $|\tilde E^1_w|=|E^1_w|-1$ (note that $r(g)\neq u,v$ for any $g\in E^1_w\setminus\{e\}$, see (i) below). By the Replacement Lemma \ref{5.4}, $L_K(E,w)\cong L_K(\tilde E, \tilde w)$. It remains to show that $L_K(\tilde E,\tilde w)$ is locally finite. By Theorem \ref{thmm1} it suffices to show that $(\tilde E,\tilde w)$ is finite, no cycle in $(\tilde E,\tilde w)$ has an exit and the weighted part of $(\tilde E,\tilde w)$ is well-behaved.\\
Clearly $(\tilde E,\tilde w)$ is finite. Next we show that no cycle in $(\tilde E,\tilde w)$ has an exit. Let $\tilde c=x_1\dots x_n$ be a cycle in $(\tilde E, \tilde w)$ and $y\in \tilde E^1$ an exit for $\tilde c$. Set 
\[\tilde E^1_{(T(v))^c\rightarrow (T(v))^c}:=\{g\mid g\in E^{1}, s(g),r(g)\in E^0\setminus T(v)\}\]
and 
\[\tilde E^1_{(T(v))^c\rightarrow (E')^0}:=\{g^{(1)},\dots,g^{(n_{r(g)})}\mid g\in E^{1}, s(g)\in E^0\setminus T(v),r(e)\in T(v)\}.\]
Then $\tilde E^1=\tilde E^1_{(T(v))^c\rightarrow (T(v))^c}\sqcup \tilde E^1_{(T(v))^c\rightarrow (E')^0}\sqcup (E')^1$. Clearly $x_1,\dots ,x_n\in \tilde E^1_{(T(v))^c\rightarrow (T(v))^c}$ or $x_1,\dots ,x_n\in (E')^1$ (note that $(E')^0$ equals the tree of $v$ in $(\tilde E,\tilde w)$ and hence is a hereditary subset of $\tilde E^0$). First assume that $x_1,\dots ,x_n\in \tilde E^1_{(T(v))^c\rightarrow (T(v))^c}$. It follows that $\tilde c$ is a cycle in $(E,w)$. Define a map $\psi: \tilde E^1_{(T(v))^c\rightarrow (E')^0}\rightarrow E^1$ by $\psi(g^{(i)})=g$. Clearly either $y\in \tilde E^1_{(T(v))^c\rightarrow (T(v))^c}$ and $y$ is an exit for $\tilde c$ in $(E,w)$ or $y\in \tilde E^1_{(T(v))^c\rightarrow (E')^0}$ and $\psi(y)$ is an exit for $\tilde c$ in $(E,w)$. But this yields a contradiction since no cycle in $(E,w)$ has an exit. Assume now that $x_1,\dots ,x_n\in (E')^1$. Then clearly $x_1, \dots, x_n,y\in s^{-1}(T(X))$. It follows that $\tilde c$ is a cycle in $(E,w)$ and $y$ is an exit for $\tilde c$. But this yields a contradiction since no cycle in $(E,w)$ has an exit.\\
It remains to show that the weighted part of $(\tilde E,\tilde w)$ is well-behaved, i.e. that Conditions (i)-(v) in Definition \ref{defwb} are satisfied.
\begin{enumerate}[(i)]
\item It clearly suffices to show $r(g)\neq u,v$ for any $g\in E^1_w\setminus\{e\}$ (and hence no weighted edge has more than one copy in the replacement graph $(\tilde E, \tilde w)$). Let $g\in E^1_w$. Then $r(g)\neq v$ since $r(g)$ is a sink. Assume now that $r(g)=u= r(e)$. It follows that $g$ and $e$ are in line since $(E,w)$ satisfies Condition (iii) in Definition \ref{defwb}. Since the ranges of $g$ and $e$ are sinks, it follows that $g=e$. Thus $r(g)\neq u,v$ for any $g\in E^1_w\setminus\{e\}$.
\item Let $\tilde g\in \tilde E^1_w$. We have to show that no vertex in $\tilde T(\tilde r(\tilde g))$ emits two distinct edges in $(\tilde E, \tilde w)$ (if $z\in \tilde E^0$, then we denote by $\tilde T(z)$ the tree of $z$ in $(\tilde E, \tilde w)$). Clearly either $\tilde g\in \tilde E^1_{(T(v))^c\rightarrow (T(v))^c}$ or $\tilde g\in\tilde E^1_{(T(v))^c\rightarrow (E')^0}$. Assume first that $\tilde g\in \tilde E^1_{(T(v))^c\rightarrow (T(v))^c}$. Then clearly $\tilde r(\tilde g)=r(\tilde g)$ is a sink in $(\tilde E,\tilde w)$ since it is a sink in $(E,w)$. Hence no vertex in $\tilde T(\tilde r(\tilde g))$ emits an edge in $(\tilde E, \tilde w)$. Assume now that $\tilde g\in\tilde E^1_{(T(v))^c\rightarrow (E')^0}$. One checks easily that $\tilde r(\tilde g)$ is a sink in $(\tilde E,\tilde w)$ if $\tilde r(\tilde g)\not \in X$. If $\tilde r(\tilde g)=x_i$ for some $1\leq i \leq m$, then $\tilde T(\tilde r(\tilde g))=\{x_i,u_{i1},\dots,u_{i, (k-1)n_i}\}$. While $x_i,u_{i1},\dots,u_{i, (k-1)n_i-1}$ emit precisely one edge in $(\tilde E,\tilde w)$, the vertex $u_{i, (k-1)n_i}$ is a sink since $x_i$ is a sink in $(E,w)$.
\item Assume that there are $\tilde g,\tilde h\in \tilde E^1_w$ such that $\tilde g$ and $\tilde h$ are not in line in $(\tilde E, \tilde w)$ (in particular $\tilde g\neq \tilde h$) and $\tilde T(\tilde r(\tilde g))\cap \tilde T(\tilde r(\tilde h))\neq \emptyset$. Clearly either $\tilde g,\tilde h\in \tilde E^1_{(T(v))^c\rightarrow (T(v))^c}$ or $\tilde g,\tilde h\in\tilde E^1_{(T(v))^c\rightarrow (E')^0}$. First assume that $\tilde g,\tilde h\in \tilde E^1_{(T(v))^c\rightarrow (T(v))^c}$. Then $\tilde r(\tilde g)$ and $\tilde r(\tilde h)$ are sinks in $(\tilde E,\tilde w)$. Hence $r(\tilde g)=\tilde r(\tilde g)=\tilde r(\tilde h)=r(\tilde h)$. Since $(E,w)$ satisfies Condition (iii) in Definition \ref{defwb}, it follows that $\tilde g$ and $\tilde h$ are in line in $(E,w)$. But that cannot hold since $\tilde g\neq \tilde h$ and $r(\tilde g)$ and $r(\tilde h)$ are sinks. Assume now that $\tilde g,\tilde h\in\tilde E^1_{(T(v))^c\rightarrow (E')^0}$. Then it is easy to see that $r(\psi(\tilde g))=\tilde r(\tilde g)=\tilde r(\tilde h)=r(\psi(\tilde h))$. Since $(E,w)$ satisfies Condition (iii) in Definition \ref{defwb}, it follows that $\psi(\tilde g)$ and $\psi(\tilde h)$ are in line in $(E,w)$. But that cannot hold since clearly $\psi(\tilde g)\neq \psi(\tilde h)$ and $r(\psi(\tilde g))$ and $r(\psi(\tilde h))$ are sinks.
\item Let $\tilde g\in \tilde E^1_w$. We have to show that no cycle in $(\tilde E, \tilde w)$ is based at a vertex in $\tilde T(\tilde r(\tilde g))$. This is obvious if $\tilde r(\tilde g)\not \in X$ since then $\tilde r(\tilde g)$ is a sink (see (ii) above). Assume now that $\tilde r(\tilde g)=x_i$ for some $1\leq i \leq m$. Then $\tilde T(\tilde r(\tilde g))=\{x_i,u_{i1},\dots,u_{i, (k-1)n_i}\}$. Obviously no cycle in $(\tilde E, \tilde w)$ is based in one of the vertices $x_i,u_{i1},\dots,u_{i, (k-1)n_i}$ (note that $u_{i, (k-1)n_i}$ is a sink since $x_i$ is a sink in $(E,w)$).
\item Assume there is an $n\geq 1$ and paths $\tilde p_1,\dots,\tilde p_n,\tilde q_1,\dots,\tilde q_n$ in $(\tilde E,\tilde w)$ such that $\tilde r(\tilde p_i)=\tilde r(\tilde q_i)~(1\leq i\leq n)$, $\tilde s(\tilde p_1)=\tilde s(\tilde q_n)$, $\tilde s(\tilde p_i)=\tilde s(\tilde q_{i-1})~(2\leq i \leq n)$ and for any $1\leq i \leq n$, the first letter of $\tilde p_i$ is a weighted edge, the first letter of $\tilde q_i$ is an unweighted edge and the last letters of $\tilde p_i$ and $\tilde q_i$ are distinct. In order to get a contradiction it suffices to show that for any $1\leq i \leq n$, 
\begin{enumerate}[(a)]
\item there are paths $p_i$ and $q_i$ in $(E,w)$ such that $s(p_i)=\tilde s(\tilde p_i)$, $r(p_i)=r(q_i)$,  $s(q_i)=\tilde s(\tilde q_i)$, the first letter of $p_i$ is a weighted edge, the first letter of $q_i$ is an unweighted edge and the last letters of $p_i$ and $q_i$ are distinct or
\item there are paths $p_i$, $q_i$, $p'_i$ and $q'_i$ in $(E,w)$ such that $s(p_i)=\tilde s(\tilde p_i)$, $r(p_i)=r(q_i)$,  $s(q_i)=s(p'_i)$, $r(p'_i)=r(q'_i)$, $s(q'_i)=\tilde s(\tilde q_i)$, the first letters of $p_i$ and $p'_i$ are weighted edges, the first letters of $q_i$ and $q'_i$ are unweighted edges, the last letters of $p_i$ and $q_i$ are distinct and the last letters of $p'_i$ and $q'_i$ are distinct
\end{enumerate}
since $(E,w)$ satisfies Condition (v) in Definition \ref{defwb}. Clearly $\tilde s(\tilde p_i)\in E^0\setminus T(v)$ for any $1\leq i\leq n$ since the first letters of the $\tilde p_i$'s are weighted edges. Hence $\tilde s(\tilde q_i)\in E^0\setminus T(v)$ for any $1\leq i\leq n$. Now fix an $i\in \{1,\dots,n\}$. If $\tilde r(\tilde p_i)=\tilde r(\tilde q_i)\in E^0\setminus T(v)$, then $p_i:=\tilde p_i$ and $q_i:=\tilde q_i$ are paths in $(E,w)$ which have the properties in (a) above (note that $(E')^0$ is hereditary in $(\tilde E, \tilde w)$). Suppose now that $\tilde r(\tilde p_i)=\tilde r(\tilde q_i)\in (E')^0$. Let $\tilde g$ be the first letter of $\tilde p_i$. Then $\tilde g\in\tilde E^1_{(T(v))^c\rightarrow (E')^0}$. Hence there is a $g\in E^1_w\setminus\{e\}$ such that $r(g)\in T(X)$ and $\tilde g=g^{(1)}$ (note that $r(g)\neq u,v$, see (i) above). \\
\\
\underline{Case 1} Assume that $r(g)=x_{i'}$ for some $1\leq i'\leq m$. Write $\tilde q_i=\tilde h^{(1)}\dots \tilde h^{(l)}$ where $\tilde h^{(1)},\dots, \tilde h^{(l)}\in \tilde E^1$. Let $t$ be maximal with the property that $\tilde s(\tilde h^{(t)})\in E^0\setminus T(v)$. Then $\tilde h^{(1)}, \dots, \tilde  h^{(t-1)} \in \tilde E^1_{(T(v))^c\rightarrow (T(v))^c}$, $\tilde h^{(t)}\in \tilde E^1_{(T(v))^c\rightarrow (E')^0}$ and $\tilde h^{(t+1)}, \dots, \tilde h^{(l)}\in (E')^1$. We distinguish four subcases.\\
\\
\underline{Case 1.1} Assume that $\tilde r(\tilde h^{(t)})\in\{u_{i''}~(1\leq i''\leq k),u_{i''j}~(1\leq i''\leq m, 1\leq j \leq (k-1)n_{i''})\}$. Set $p_i:=g$, $q_i=f^{(i'1)}$, $p_i':=e$ and $q'_i:=\tilde h^{(1)} \dots \tilde h^{(t-1)}\psi(\tilde h^{(t)})$. One checks easily that $p_i$, $q_i$, $p'_i$ and $q'_i$ have the properties in (b) above.\\
\\
\underline{Case 1.2} Assume that $\tilde r(\tilde h^{(t)})\in\{v, v_{i''j}~(1\leq i''\leq m, 2\leq j \leq n_{i''})\}$. Set $p_i:=g$ and $q_i=\tilde h^{(1)} \dots \tilde h^{(t-1)}\psi(\tilde h^{(t)})f^{(i'1)}$. One checks easily that $p_i$ and $q_i$ have the properties in (a) above. \\
\\
\underline{Case 1.3} Assume that $\tilde r(\tilde h^{(t)})=x_{i'}$. Set $p_i:=g$ and $q_i=\tilde h^{(1)} \dots \tilde h^{(t-1)}\psi(\tilde h^{(t)})$. One checks easily that $p_i$ and $q_i$ have the properties in (a) above.\\
\\
\underline{Case 1.4} Assume that $\tilde r(\tilde h^{(t)})\in T(X)\setminus X$. Then clearly $\tilde r(\tilde p_i)=\tilde r(\tilde q_i)=u^{(i',(k-1)n_{i'})}$. Set $p_i:=g$ and $q_i=\tilde h^{(1)} \dots \tilde h^{(t-1)}\psi(\tilde h^{(t)})\tilde h^{(t+1)} \dots \tilde h^{(l)}$. One checks easily that $p_i$ and $q_i$ have the properties in (a) above.\\
\\
\underline{Case 2} Assume that $r(g)\in T(X)\setminus X$. This case is similar to Case 1 and is left to the reader.
\end{enumerate}
\end{proof}

Now we are ready to prove the main result of this section.
\begin{theorem}\label{thm5}
If $L_K(E,w)$ is locally finite, then it is isomorphic to a locally finite Leavitt path algebra. It follows that $L_K(E,w)\cong(\bigoplus\limits_{i=1}^{l} M_{m_i}(K))\oplus(\bigoplus\limits_{j=1}^{l'} M_{n_j}(K[x,x^{-1}]))$ for some integers $l,l'\geq 0$ and $m_i,n_j\geq 1$. Hence $L_K(E,w)$ is Noetherian.
\end{theorem}
\begin{proof}
It follows from Propositions \ref{5.2} and \ref{5.5} that there is an unweighted weigthed graph $(\tilde E,\tilde w)$ such that $L_K(\tilde E,\tilde w)\cong L_K(E,w)$ and $L_K(\tilde E,\tilde w)$ is locally finite. Hence, by Example \ref{exex1} and \cite[Theorem 4.2.17]{abrams-ara-molina}, $L_K(E,w)$ is isomorphic to $(\bigoplus\limits_{i=1}^{l} M_{m_i}(K))\oplus(\bigoplus\limits_{j=1}^{l'} M_{n_j}(K[x,x^{-1}]))$ for some integers $l,l'\geq 0$ and $m_i,n_j\geq 1$. It is well-known that $K[x,x^{-1}]$ is a Noetherian ring and hence so is any finite matrix ring over $K[x,x^{-1}]$. Thus $L_K(E,w)$ is Noetherian.\end{proof}

\begin{example}
Suppose $(E,w)$ is the weighted graph 
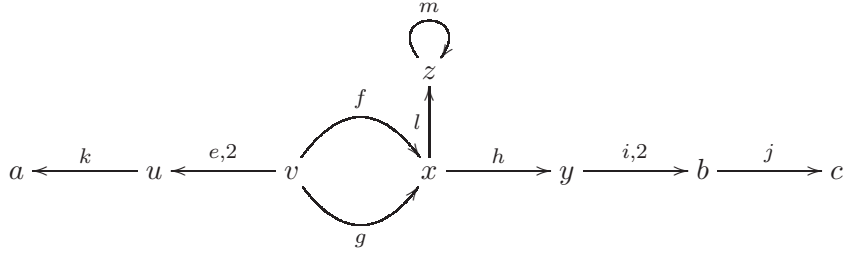
\[
\xymatrix@C+15pt{
&&&z\ar@(ul,ur)^{m}&&&\\
a&u\ar[l]_{k}&v\ar[l]_{e,2}\ar@/^1.7pc/[r]^{f}\ar@/_1.7pc/[r]_{g}&x\ar[u]^{l}\ar[r]^{h}&y\ar[r]^{i,2}&b\ar[r]^{j}&c}
\]
from Example \ref{ex5.00}. Let $(\tilde E, \tilde w)$ be the unweighted weighted graph
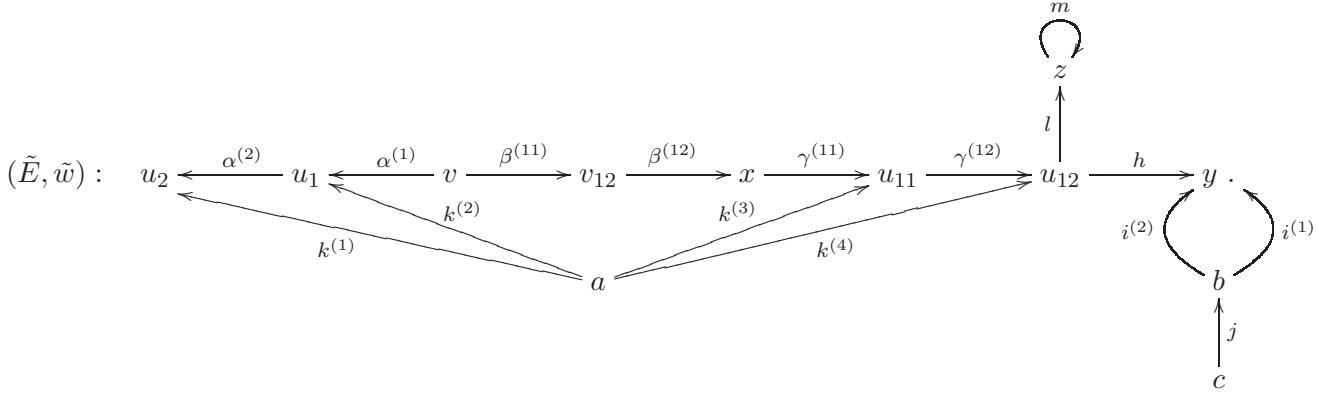
\[
\quad
\xymatrix@C+15pt{&&&&&&z\ar@(ul,ur)^{m}&\\
(\tilde E, \tilde w):\quad u_2&u_1\ar[l]_{\hspace{1.2cm}\alpha^{(2)}}&v\ar[l]_{\hspace{0.5cm}\alpha^{(1)}}\ar[r]^{\beta^{(11)}}&v_{12}\ar[r]^{\beta^{(12)}}&x\ar[r]^{\gamma^{(11)}}&u_{11}\ar[r]^{\gamma^{(12)}}&u_{12}\ar[u]^{l}\ar[r]^{h}&y~.\\
&&&a\ar[lllu]^{k^{(1)}}\ar[llu]_{k^{(2)}}\ar[rru]^{k^{(3)}}\ar[rrru]_{k^{(4)}}&&&&b\ar@/^1.7pc/[u]^{i^{(2)}}\ar@/_1.7pc/[u]_{i^{(1)}}\\
&&&&&&&c\ar[u]_{j}}
\]
Then $L_K(E,w)\cong L_K(\tilde E, \tilde w)$ by Examples \ref{ex5.00} and \ref{ex5.2}. It follows from \cite[Corollary 4.2.14]{abrams-ara-molina} that $L_K(E,w)\cong M_5(K)\times M_{12}(K)\times M_8(K[x,x^{-1}])$.
\end{example}

\section{Noetherian wLpas are locally finite}
In this section we show that if $L_K(E,w)$ is Noetherian, then it is locally finite. Note that $L_K(E,w)$ is a ring with involution (cf. \cite[Proposition 5.7]{hazrat13}) and hence 
\[L_K(E,w)\text{ is left Noetherian }\Leftrightarrow~L_K(E,w)\text{ is right Noetherian }\Leftrightarrow~L_K(E,w)\text{ is Noetherian}.\]

\begin{definition}[{\sc Lenod-path}]
A {\it left-normal d-path} or {\it lenod-path} is a nod-path $p$ such that the juxtaposition $op$ is a nod-path for any nontrivial nod-path $o$ such that $r(o)=s(p)$.
\end{definition}

If $A, B$ are nonempty words over the same alphabet, then we write $A\sim B$ if $A$ is a suffix of $B$ or $B$ is a suffix of $A$ and $A\not\sim B$ otherwise. If $X\subseteq L_K(E,w)$, then we denote by $LI(X)$ the left ideal of $L_K(E,w)$ generated by $X$.

\begin{lemma}\label{lem4.1}
If $L_K(E,w)$ is Noetherian, then there is no lenod-path $p$ and nod$^2$-path $q$ such that $p\not\sim q$ and $pq$ is a nod-path.
\end{lemma}
\begin{proof}
Assume that there is a lenod-path $p$ and a nod$^2$-path $q$ such that $p\not\sim q$ and $pq$ is a nod-path. Let $n\geq 0$. Since $p$ is a lenod-path, the left ideal $LI(p,pq,\dots,pq^n)$ equals the linear span of all nod-paths $o$ such that one of the words $p,pq,\dots,pq^n$ is a suffix of $o$. It follows that 
$LI(p,pq,\dots,pq^n)\subsetneq LI(p,pq,\dots,pq^{n+1})$ (clearly none of the words $p,pq,\dots,pq^n$ is a suffix of $pq^{n+1}$ since $p\not\sim q$) and thus $L_K(E,w)$ is not Noetherian.
\end{proof}

\begin{corollary}\label{cor4.1}
If $L_K(E,w)$ is Noetherian, then there is no nod-path $q=x_1\dots x_n$ such that $x_1=e_2$ and $x_n=e_2^*$ for some $e\in E^1_w$.
\end{corollary}
\begin{proof}
Assume that there is a nod-path $q=x_1\dots x_n$ such that $x_1=e_2$ and $x_n=e_2^*$ for some $e\in E^1_w$. Then clearly $q$ is a nod$^2$-path. Set $p:=e_2x_2\dots x_{n-1}e_1^*$. One checks easily that $p$ is a lenod-path such that $p\not\sim q$ and $pq$ is a nod-path. Thus, by Lemma \ref{lem4.1}, $L_K(E,w)$ is not Noetherian.
\end{proof}

\begin{lemma}\label{lem4.2}
If $L_K(E,w)$ is Noetherian, then there is no nod$^2$-path $p$ based at a vertex $v$ such that $pp^*$ is a nod-path and $p^*p=v$ in $L_K(E,w)$.
\end{lemma}
\begin{proof}
Assume that there is a nod$^2$-path $p$ based at a vertex $v$ such that $pp^*$ is a nod-path and $p^*p=v$ in $L_K(E,w)$. Let $n\geq 0$. One checks easily that $(v-p^n(p^*)^n)(v-p^{n+1}(p^*)^{n+1})=v-p^n(p^*)^n$. Hence $LI(v-p^n(p^*)^n)\subseteq LI(v-p^{n+1}(p^*)^{n+1})$. Assume now that there is an $a\in L_K(E,w)$ such that $a(v-p^n(p^*)^n)=v-p^{n+1}(p^*)^{n+1}$. Multiplying the equation by $p^n$ on the right we get $0=p^n-p^{n+1}p^*$ which cannot hold since $p^n-p^{n+1}p^*$ is a nontrivial linear combination of nod-paths. Hence $LI(v-p^n(p^*)^n)\subsetneq LI(v-p^{n+1}(p^*)^{n+1})$ and thus $L_K(E,w)$ is not Noetherian.
\end{proof}

\begin{theorem}\label{thm4}
If $L_K(E,w)$ is Noetherian, then it is locally finite.
\end{theorem}
\begin{proof}
By Theorem \ref{thmm1} it suffices to show that $(E,w)$ is finite, no cycle has an exit and the weighted part of $(E,w)$ is well-behaved. Clearly $(E,w)$ is finite. Otherwise, taking a sequence $(v_n)_{n\geq 1}$  of pairwise distinct vertices one would have $LI(v_1)\subsetneq LI(v_1,v_2)\subsetneq\dots$. Next we will show that the weighted part of $(E,w)$ is weakly well-behaved.\\
One checks easily that if one of the Conditions (i), (ii) and (iii) in Definition \ref{defwb} is not satisfied, then there is a nod-path $q=x_1\dots x_n$ such that $x_1=e_2$ and $x_n=e_2^*$ for some $e\in E^1_w$ (cf. \cite[Proof of Lemma 33]{preusser}) which contradicts Corollary \ref{cor4.1}. Assume now that Condition (iv) in Definition \ref{defwb} is not satisfied. Then there is an $e\in E^1_w$, an $o$ which is either the empty word or a path in $\hat E$ and a cycle $c$ in $\hat E$ such that $e_2oc$ is a nod-path. Clearly we may assume that none of the edges of $c$ appears in $o$. If $e_2o\not\sim c$, then we can apply Lemma \ref{lem4.1} to get a contradiction. If $e_2o\sim c$, then $o$ must be the empty word and $c=x_1\dots x_{n-1} e_2$. But then we can replace $c$ by the cycle $c':=x_1\dots x_{n-1} e_1$ and apply Lemma \ref{lem4.1} to get a contradiction. Hence Condition (iv) is satisfied and therefore the weighted part of $(E,w)$ is weakly well-behaved.\\
It remains to show that Condition (v) in Definition \ref{defwb} is satisfied and that no cycle has an exit. Assume that Condition (v) is not satisfied. Then there is a nod$^2$-path $o=x_1\dots x_m$ such that $x_1=e_2$ for some $e\in E^1_w$. By Lemma \ref{lem3.-1}, $o=p_1q_1^*\dots p_nq_n^*$ or $o=p_1q_1^*\dots p_{n-1}q_{n-1}^*p_n$ where $n\geq 1$, $p_1,\dots,p_n$ are super-special paths in $\hat E$ and $q_1,\dots,q_n$ are unweighted paths in $\hat E$. Set $v:=s(e)$. One checks easily that $o^*o$ is a nod-path and $oo^*=v$ in $L_K(E,w)$ (note that $p_ip_i^*=s(p_i)$ and $q^*_iq_i=r(q_i)$ for any $i$ since the $p_i$'s are super-special and the $q_i$'s are unweighted). But that contradicts Lemma \ref{lem4.2}. Hence Condition (v) is satisfied and therefore the weighted part of $(E,w)$ is well-behaved.\\
Assume now that there is a cycle $c=e^{(1)}\dots e^{(n)}$ in $(E,w)$ with an exit $f\in E^1$. Clearly $e^{(1)},\dots, e^{(n)}\in E^1_u$ because of Condition (iv) in Definition \ref{defwb}. W.l.o.g. assume that $s(f)=s(e^{(n)})$ and $f\neq e^{(n)}$. Clearly we can choose $e^{s(e^{(n)})}\neq e^{(n)}$ since $w(e^{(n)})=1$. Set $v:=s(c)$ and $\hat c:=e^{(1)}_1\dots e^{(n)}_1$. Then clearly $\hat c$ is a nod$^2$-path based at $v$, $\hat c\hat c^*$ is a nod-path and $\hat c^*\hat c=v$. But that contradicts Lemma \ref{lem4.2}. Hence no cycle in $(E,w)$ has an exit. 
\end{proof}

\section{Summary}
By Theorems \ref{thmm1}, \ref{thm5} and \ref{thm4} we have the following result:
\begin{theorem}\label{thmm}
Let $K$ be a field and $(E,w)$ a row-finite weighted graph. Then the following are equivalent:
\begin{enumerate}[(i)]
\item $L_K(E,w)$ is locally finite.
\item $(E,w)$ is finite, no cycle has an exit and the weighted part of $(E,w)$ is well-behaved.
\item $(E,w)$ is finite and $\GKdim L_K(E,w)\leq 1$.
\item $L_K(E,w)$ is isomorphic to a locally finite Leavitt path algebra.
\item $L_K(E,w)\cong(\bigoplus\limits_{i=1}^{l} M_{m_i}(K))\oplus(\bigoplus\limits_{j=1}^{l'} M_{n_j}(K[x,x^{-1}]))$ for some integers $l,l'\geq 0$ and $m_i,n_j\geq 1$.
\item $L_K(E,w)$ is Noetherian.
\end{enumerate}
\end{theorem}

\end{document}